\documentclass{amsart}
\usepackage{verbatim}
\usepackage[textsize=scriptsize]{todonotes}
\usepackage{longtable}
\usepackage{bbm}
\usepackage{amscd}
\usepackage{stmaryrd}

\usepackage[margin=1in]{geometry}
\usepackage{verbatim}
\usepackage[textsize=scriptsize]{todonotes}
\usepackage{tikz-cd}
\usepackage{etoolbox}
\usepackage{etex}
\usepackage[pdfusetitle,unicode,hidelinks]{hyperref}
\usepackage{cleveref}
\usepackage[T1]{fontenc}

\usepackage{chemarr}
\usepackage{amssymb}
\usepackage{amsmath}
\usepackage{comment}
\usepackage{spectralsequences}
\usepackage{cleveref}
\usepackage{mathtools}
\usepackage{rotating}
\usepackage{wrapfig}
\usepackage{outlines}
\usepackage{graphicx}
\usepackage{scalerel}

\numberwithin{equation}{section}

\theoremstyle{definition}

\newtheorem{rmk}[equation]{Remark}

\newtheorem{cnstr}[equation]{Construction}

\newtheorem{exm}[equation]{Example}

\newtheorem*{dfn*}{Definition}
\newtheorem*{axm*}{Axiom}
\newtheorem*{ntn*}{Notation}
\newtheorem*{exm*}{Example}
\newtheorem*{exr*}{Exercise}
\newtheorem*{int*}{Intuition}
\newtheorem*{qst*}{Question}
\newtheorem*{rmk*}{Remark}

\theoremstyle{plain}
\newtheorem{sch}[equation]{Scholium}

\newtheorem{thm}[equation]{Theorem}
\newtheorem{prop}[equation]{Proposition}
\newtheorem{lem}[equation]{Lemma}

\newtheorem{cor}[equation]{Corollary}

\newtheorem*{thm*}{Theorem}
\newtheorem*{prop*}{Proposition}
\newtheorem*{cor*}{Corollary}
\newtheorem*{lem*}{Lemma}
\newtheorem*{cnj*}{Conjecture}

\let\oldwidetilde\widetilde
\protected\def\widetilde{\oldwidetilde}

\DeclareMathOperator{\Map}{\mathrm{Map}}

\DeclareMathOperator{\Hom}{\mathrm{Hom}}

\newcommand{\Ss}{\mathbb{S}}

\newcommand{\F}{\mathbb{F}}

\newcommand{\NN}{\mathbb{N}}
\newcommand{\A}{\mathbb{A}}

\DeclareMathOperator{\MU}{\mathrm{MU}}

\DeclareMathOperator{\Spec}{\text{Spec}}

\DeclareMathOperator{\Pic}{\mathrm{Pic}}
\DeclareMathOperator{\Fin}{\mathrm{Fin}}
\DeclareMathOperator{\Ext}{\mathrm{Ext}}

\DeclareMathOperator{\Fun}{\mathrm{Fun}}
\DeclareMathOperator{\Alg}{\mathrm{Alg}}

\newcommand{\BP}{\mathrm{BP}}

\newcommand{\Z}{\mathbb{Z}}

\DeclarePairedDelimiter\abs{\lvert}{\rvert}%
\makeatletter
\let\oldabs\abs
\def\abs{\@ifstar{\oldabs}{\oldabs*}}

\usepackage{tikz}
\usetikzlibrary{matrix,arrows,decorations}
\usepackage{tikz-cd}

\usepackage{adjustbox}

\let\oldtocsection=\tocsection
 
\let\oldtocsubsection=\tocsubsection
 
\let\oldtocsubsubsection=\tocsubsubsection
 
\renewcommand{\tocsection}[2]{\hspace{0em}\oldtocsection{#1}{#2}}
\renewcommand{\tocsubsection}[2]{\hspace{1em}\oldtocsubsection{#1}{#2}}
\renewcommand{\tocsubsubsection}[2]{\hspace{2em}\oldtocsubsubsection{#1}{#2}}

\let\scr=\mathcal

\def\comp{\wedge}
\newcommand{\MGL}{\mathrm{MGL}}
\newcommand{\BPGL}{\mathrm{BPGL}}
\newcommand{\PMGL}{\mathrm{PMGL}}
\newcommand{\KGL}{\mathrm{KGL}}

\newcommand{\SH}{\mathcal{SH}}
\newcommand{\DA}{\mathcal{DA}}
\newcommand{\Spc}{\mathcal{S}\mathrm{pc}}
\newcommand{\et}{{\acute{e}t}}
\newcommand{\ret}{{r\acute{e}t}}
\newcommand{\wequi}{\simeq}
\newcommand{\1}{\mathbbm{1}}
\def\map{\mathrm{map}}
\DeclareRobustCommand{\ul}{\underline}

\def\PSh{\mathcal{P}}
\newcommand{\Shv}{\mathcal{S}\mathrm{hv}}
\def\ph{\mathord-}
\def\op{\mathrm{op}}
\newcommand{\Ab}{\mathrm{Ab}}
\def\Cat{\mathcal{C}\mathrm{at}{}}

\def\CAlg{\mathrm{CAlg}}
\def\adj{\rightleftarrows}
\def\Sm{{\mathcal{S}\mathrm{m}}}
\def\SmQP{{\mathcal{S}\mathrm{mQP}}}
\def\SmQA{{\mathcal{S}\mathrm{mQA}}}
\def\SmAff{{\mathcal{S}\mathrm{mAff}}}
\def\Sch{{\mathcal{S}\mathrm{ch}}}

\def\FEt{\mathrm{FEt}{}}

\newcommand{\Mod}{\mathrm{Mod}}

\def\P{\mathbb P}
\def\CMon{\mathrm{CMon}}

\def\PrL{\mathrm{Pr}^\mathrm{L}}

\def\Nis{\mathrm{Nis}}
\def\Zar{\mathrm{Zar}}

\def\mot{\mathrm{mot}}

\newcommand{\lra}[1]{\langle #1 \rangle}

\newcommand{\gp}{\mathrm{gp}}
\def\Vect{\mathcal{V}\mathrm{ect}{}}

\input cyracc.def
\DeclareFontFamily{U}{russian}{}
\DeclareFontShape{U}{russian}{m}{n}
        { <5><6> wncyr5
        <7><8><9> wncyr7
        <10><10.95><12><14.4><17.28><20.74><24.88> wncyr10 }{}
\DeclareSymbolFont{Russian}{U}{russian}{m}{n}
\DeclareSymbolFontAlphabet{\mathcyr}{Russian}
\makeatletter
\let\@math@cyr\mathcyr
\renewcommand{\mathcyr}[1]{\@math@cyr{\cyracc #1}}
\makeatother

\newcommand{\Cor}{\mathrm{Cor}}
\def\fet{\mathrm{f\acute et}}
\def\all{\mathrm{all}}
\def\Span{\mathrm{Span}}
\def\BiSpan{\mathrm{BiSpan}}

\def\FFree{\mathrm{FFree}}
\def\h{\mathrm h}
\def\NAlg{\mathrm{NAlg}}

\newtoggle{final}
\toggletrue{final}

\iftoggle{final} {
\renewcommand{\todo}[1]{}
\newcommand{\NB}[1]{}
}{ 
\newcommand{\NB}[1]{\todo[color=gray!40]{#1}}
}

\newcommand{\TODO}[1]{\todo[color=red]{#1}}

\title{Motivic stable homotopy theory is strictly commutative at the characteristic}
\date{\today}

\author{Tom Bachmann}
\address{Mathematischen Institut, LMU Munich, Germany}
\email{tom.bachmann@zoho.com}

\begin{document}

\maketitle

\begin{abstract}
We show that mapping spaces in the $p$-local motivic stable category over an $\F_p$-scheme are strictly commutative monoids (whence $H\Z$-modules) in a canonical way.
\end{abstract}

\setcounter{tocdepth}{1}
\tableofcontents

\section{Introduction}
It is well-known that ``motivic invariants at the characteristic'' behave in strange ways.
One famous example is Quillen's computation of the $K$-theory of finite fields \cite{quillen1972cohomology}.
He shows in particular that $K(\F_p)_{(p)} \wequi H\Z_{(p)}$.
If $S$ is any $\F_p$-scheme, then $K(S)$ is a module over $K(\F_p)$, and so $K(S)_{(p)}$ is a module over $H\Z_{(p)}$.
This implies for example that all $k$-invariants of $K(S)_{(p)}$ vanish, in stark contrast with the situation away from the characteristic (i.e. for $K(S)[1/p]$).

In this article we establish the following generalization: if $S$ is any $\F_p$-scheme, and $E \in \SH(S)$ is any motivic spectrum over $S$, then $E(S)_{(p)}$ is an $H\Z_{(p)}$-module.

We begin in \S\ref{sub:intro1} with a rapid recollections of the definitions necessary to state precisely our main result.
The following \S\ref{sub:intro2} contains a more leisurely discussion of background and motivation.
Depending on their taste, the reader may wish to reverse the order of reading these first two subsections.

\subsection{Main result} \label{sub:intro1}
Let $\scr C$ be a stable, presentably symmetric monoidal $\infty$-category.
Then there is a unique cocontinuous symmetric monoidal functor $c: \PSh_\Sigma(\Span(\Fin)) \to \scr C$ (here $\PSh_\Sigma$ denotes the nonabelian derived category  \cite[\S5.5.8]{HTT}, and $\Span(\Fin)$ denotes the $(2,1)$-category of spans in finite sets); it admits a right adjoint $c^*$ by presentability.
Let $X, Y \in \scr C$.
Denoting by $M(X, Y) \in \scr C$ the internal mapping object, we obtain $c^*M(X,Y) \in \PSh_\Sigma(\Span(\Fin))$.
One has $c^*(M(X,Y))(*) \wequi \Map_{\scr C}(X, Y)$.
Recalling that $\PSh_\Sigma(\Span(\Fin))$ is the category of commutative monoids in spaces \cite[Proposition C.1]{norms}, the functor $c$ encodes the fact that mapping spaces in $\scr C$ are commutative monoids in a canonical way.

Observe that $\h\Span(\Fin)$ (the homotopy category of $\Span(\Fin)$) is canonically equivalent to the symmetric monoidal $1$-category $\FFree_\NN$ of finitely generated, free commutative monoids in sets.
Objects of $\PSh_\Sigma(\FFree_\NN)$ are called \emph{strictly commutative monoids}.
The subcategory of grouplike objects identifies with the category $\PSh_\Sigma(\FFree_\Z)$, also known as connective $H\Z$-modules.
The following is our main result.
It says that mapping spaces in the $p$-local motivic stable category of an $\F_p$-scheme are strictly commutative monoids (and even connective $H\Z$-modules) in a canonical way.
\begin{thm}[see Corollary \ref{cor:main-result}] \label{thm:main}
Let $S$ be an $\F_p$-scheme.
The canonical symmetric monoidal functor $\Span(\Fin) \to \SH(S)_{(p)}$ (where $\SH(S)_{(p)} \subset \SH(S)$ denotes the subcategory of $p$-local objects) factors canonically (as a symmetric monoidal functor) through $\FFree_\NN$ (and even $\FFree_\Z$ and $\FFree_{\Z_{(p)}}$).
\end{thm}
\begin{exm}
We deduce that $K(\F_p)_{(p)} = \Map_{\SH(\F_p)}(\1, \KGL_{(p)})$ is an $H\Z$-module.
This is well-known; in fact $K(\F_p)_{(p)} \wequi H\Z_{(p)}$ by \cite{quillen1972cohomology}.
\end{exm}

\begin{exm}
The category $\DA(S)$ is given by modules over the image of $\Z$ under the canonical functor $\PSh_\Sigma(\Span(\Fin)) \to \SH(S)$.
Consequently if $S$ has characteristic $p$, $\DA(S)_{(p)}$ is given by modules over the image of $H\Z_{(p)} \otimes H\Z_{(p)}$ under $\PSh_\Sigma(\FFree_\Z) \to \SH(S)$.
In particular $\DA(S)_{(p)}$ admits $\SH(S)_{(p)}$ as a (proper) factor.
\end{exm}

One naturally expects such higher structure to have plenty of applications.
Here are two quick ones.
\begin{exm}
The Moore spectrum $\1/p \in \SH(S)$ is an $\scr E_\infty$-ring.
Indeed it is the image of the $\scr E_\infty$-ring $H\Z/p \in \PSh_\Sigma(\FFree_\Z)$.
\end{exm}
\begin{cor}
Let $S$ be an $\F_p$-scheme.
Denote by $\SH_\et^{nh}(S) \subset \SH(S)$ the subcategory of spectra satisfying étale descent, by $\SH_\et^\comp(S)$ the subcategory of spectra satisfying étale hyperdescent, and by $(\ph)_p^\comp$ the subcategory of $p$-complete objects.
Then $\SH_\et^{nh}(S)_p^\comp = * = \SH_\et^{nh}(S)_p^\comp$.
\end{cor}
\begin{proof}
It suffices to prove that $\SH_\et^{nh}(\F_p)_p^\comp = *$, for which by \cite[Theorem A.1]{bachmann-SHet2} it suffices to prove that $\SH_\et^{nh}(\F_p)_p^\comp = \SH_\et^\comp(\F_p)_p^\comp$.
This follows from Theorem \ref{thm:main}, since $\Z$-linear étale sheaves on $\Sm_{\F_p}$ are hypercomplete \cite[Proposition 2.31]{clausen2019hyperdescent}.
\end{proof}

\begin{rmk}
We have a variant of Theorem \ref{thm:main} (see Corollary \ref{cor:SCR}) showing that if $E \in \NAlg(\SH(S)_{(p)})$ is a $p$-local \emph{normed motivic spectrum} \cite[\S7]{norms}, then $\Map_{\SH(S)}(\1, E)$ is a simplicial commutative ring.
There are also several other variants and complements.
\end{rmk}

\subsection{Background} \label{sub:intro2}
The Hopkins--Morel problem asks if the canonical map of motivic spectra\footnote{Beware that we use the symbol $H\Z$ to denote both the classical and the motivic Eilenberg--MacLane spectrum. Since the two live in different categories, this should not cause confusion.} (arising from the fact that $H\Z$ is orientable with additive formal group law) \[ \MGL/(p, t_1, t_2, \dots) \to H\Z/p \in \SH(S) \] is an equivalence.
By work of Hopkins--Morel--Hoyois \cite{hoyois2015algebraic} and Spitzweck \cite{spitzweck2012commutative}, this is true whenever $p$ is invertible on $S$.
Thus the problem has an affirmative solution in general if and only if the map is an equivalence in $\SH(\F_p)$.

Assuming the equivalence in $\SH(\F_p)$, a slice spectral sequence computation (using \cite{spitzweck2010relations}) together with the work of Geisser--Levine \cite{geisser2000k} shows that\footnote{Where now $E_p^\comp$ also denotes the $p$-completion of an object in a stable $\infty$-category.} \[ \pi_{**}((\MGL_{k})_p^\comp) \wequi K_*^M(k)_p^\comp \otimes L_* \] for any field $k$ of characteristic $p$ (here $K_*^M(k)$ denotes the Milnor $K$-theory of $k$, $K_n^M$ is placed in bidegree $(-n,-n)$, $L_*$ denotes the Lazard ring and $L_{2n}$ is placed in bidegree $(2n,n)$).
From this one deduces that the strongly convergent motivic Adams--Novikov spectral sequence collapses ($p$-adically) over $\F_p$ and hence finds that \[ \pi_{2w-s,w} ((\1_{\F_p})_p^\comp) \wequi \Ext^{s,2w}_{\MU_*\MU}(\MU_*)_p^\comp. \]
In particular \[ \pi_{*,0}((\1_{\F_p})_p^\comp) \wequi \begin{cases} \Z_p & *=0 \\ 0 & \text{else} \end{cases}. \]
Consequently $\map_{\SH(\F_p)_p^\comp}(\1, \1) \wequi H\Z_p$, which immediately implies that $\SH(\F_p)_p^\comp$ is $\Z_p$-linear.

One can deduce many further results from this assumption.
For example it follows that $\Pic(\SH(\F_p)_p^\comp) \in \CMon(\Spc)^\gp$ is $1$-truncated.
Thus given an invertible object $E \in \SH(\F_p)_p^\comp$, there exists a corresponding map $H\Z \xrightarrow{\eta_E} \Pic(\SH(\F_p)_p^\comp)$ if and only if the switch map on $E^{\otimes 2}$ is homotopic to the identity.
We verify in Corollary \ref{cor:HZ-Pic} that $\eta_E$ exists unconditionally in the special case of $E = \Sigma^{2,1} \1_p^\comp$ (and $p \ne 2$).

In other words, an affirmative solution of the Hopkins--Morel problem would imply many surprising results about $\SH(\F_p)_p^\comp$, some of which we establish unconditionally in this article.
I view this as small (further) evidence towards the Hopkins--Morel problem.

\subsection{Why is the main result true?}
In proving Theorem \ref{thm:main}, we exploit the fact that $\SH(\ph)$ is a \emph{functor}, i.e. that we can vary the base scheme.
In fact we do this in two ways.

Firstly, if $f: S' \to S$ is any morphism of schemes, the functor $f^*: \SH(S) \to \SH(S')$ admits a right adjoint $f_*$.
If $f$ is smooth, there is also a left adjoint $f_\sharp$.
If $f: S' = S \amalg S \to S$ is the fold map, then $\SH(S \amalg S) \wequi \SH(S) \times \SH(S)$, the functor $f_*$ is the binary product, and $f_\sharp$ is the binary sum.
Since $\SH(S)$ is semiadditive, there is thus a canonical equivalence $f_\sharp \wequi f_*$ (for this $f$).
This is true more generally if $f$ is a \emph{finite étale morphism}, by a property of $\SH(\ph)$ called \emph{ambidexterity} (see e.g. \cite[Theorem 6.18(1,2)]{hoyois-equivariant}).
Just as semiadditivity yields a canonical functor $\Span(\Fin) \to \SH(S)$, ambidexterity yields a canonical functor $\Span(\FEt_S) \to \SH(S)$ (here $\FEt_S$ denotes the category of finite étale $S$-schemes).

Secondly, if $\scr C: \Sm_S^\op \to \Cat_\infty$ is a functor, then mapping spaces in $\scr C(S)$ are automatically enriched in $\PSh(\Sm_S)$.
Namely given $X, Y \in \scr C(S)$ we obtain \[ \PSh(\Sm_S) \ni \ul\Map(X, Y): S' \mapsto \Map_{\scr C(S')}(X_{S'}, Y_{S'}). \]
For example if $\scr C(\ph) = \Span(\FEt_{(\ph)})$, then $\ul\Map(S,S) \wequi \FEt_{(\ph)}^\wequi$ is the presheaf of groupoids of finite étale schemes.
For $\scr C = \SH(\ph)_{(p)}$, the mapping presheaves are automatically local objects for a certain localization $L$ of $\PSh(\Sm_S)$ (namely the localization with equivalences detected by $\Sigma^\infty_+(\ph)_{(p)}: \PSh(\Sm_S) \to \SH(S)_{(p)}$).
It turns out that one can form a category enriched in $\PSh(\Sm_S)$ by $L$-localizing the mapping presheaves in $\Span(\FEt_{(\ph)})$.
Denote the usual (non-enriched) $\infty$-category obtained from this by taking global sections of the mapping presheaves as $L\Span(\FEt_S)$.
We have thus obtained a factorization \[ \Span(\Fin) \to \Span(\FEt_S) \to L\Span(\FEt_S) \to \SH(S)_{(p)}. \]
The main observation is that if $S$ is an $\F_p$-scheme, then mapping presheaves in $L\Span(\FEt_S)$ can be computed and are in fact discrete; this is a generalization of the folklore result that the étale classifying space of a $p$-group is motivically contractible over $\F_p$ \cite[Remark 4.3.4]{A1-homotopy-theory}.
Thus $\Span(\Fin) \to L\Span(\FEt_S)$ factors through $\h\Span(\Fin)$, proving Theorem \ref{thm:main}.

\subsection{Organization}
We prove the contractibility result for étale classifying spaces alluded to above in \S\ref{sec:etale-class-spaces}.
Then in \S\ref{sec:enriched} we recall basic facts about enriched $\infty$-categories from \cite{gepner2015enriched}.
Next in \S\ref{sec:semiadd} we recall some universality properties of $\Span(\Fin)$ with respect to semiadditive $\infty$-categories, and then we explain some motivic analogs.
In particular we construct the canonical functor $\Span(\FEt_S) \to \SH(S)$ mentioned above.
Finally in \S\ref{sec:main} we prove our main result, following the sketch above.
We establish some complements and variants in \S\ref{sec:complements}; this long section has its own introduction.

\subsection{Notation and conventions}
We use freely the language of $\infty$-categories as established in \cite{HTT,HA}, and also assume basic familiarity with unstable and stable motivic homotopy theory (see e.g. \cite[\S\S2.2, 4.1]{norms}).
Given a stable $\infty$-category $\scr C$, we denote by $\scr C_{(p)}$ the subcategory of $p$-local objects (i.e. those such that the endomorphism of multiplication by $n$ is an equivalence, for every $n$ coprime to $p$).
Given a small $\infty$-category $\scr C$ with finite coproducts we denote by $\PSh_\Sigma(\scr C)$ its nonabelian derived category \cite[\S5.5.8]{HTT}, and by $L_\Sigma: \PSh(\scr C) \to \PSh_\Sigma(\scr C)$ the associated localization.

\subsection{Acknowledgements}
I would like to thank Rune Haugseng for pointing me towards his work on $\infty$-categories enriched in absolute distributors, which significantly simplified Proposition \ref{prop:PSh-enriched}.
I would also like to thank the anonymous referee for their careful reading and several suggestions for improvement.

\section{Étale classifying spaces} \label{sec:etale-class-spaces}
Let $S$ be a scheme and $G$ be a presheaf of groups on $\Sm_S$.
We obtain a presheaf of spaces $BG \in \PSh(\Sm_S)$ and denote by $B_\et G := L_\et BG \in \PSh(\Sm_S)$ its étale localization.
If $G$ is abelian, we similarly have $B^n_\et G \in \PSh(\Sm_S)$ for $n \ge 2$.

The following result is essentially folklore \cite[Remark 4.3.4]{A1-homotopy-theory}.
\begin{prop} \label{prop:Bet}
Let $S$ be an $\F_p$-scheme and $G$ an ind-(finite étale) group scheme over $S$.
\begin{enumerate}
\item If $G$ is a $p$-group, then $L_\mot B_\et G \wequi * \in \Spc(S)$.
\item If $G$ is an abelian $p$-group, then $L_\mot B^n_\et G \wequi * \in \Spc(S)$ also for $n > 1$.
\item Without further assumptions on $G$, the projection map $B_\et G \to *$ becomes an equivalence upon applying the functor $\Sigma^\infty_+(\ph)_{(p)}: \PSh(\Sm_S) \to \SH(S)_{(p)}$.
\end{enumerate}
\end{prop}
\begin{proof}
By Zariski descent, we may assume $S$ affine.
By continuity of étale cohomology, we may assume $G$ finite étale, and then (étale classifying spaces being stable under base change, e.g. by \cite[\S4.2]{A1-homotopy-theory}) also $S$ finite type over $\F_p$.\NB{finite étale implies finitely presented}
Thus $S$ is locally connected, and we may assume without loss of generality that $S$ is connected.\NB{Not that it really matters.}
We can replace $\PSh(\Sm_S)$ by $\PSh(\SmAff_S)$ throughout, since the categories become equivalent after (Zariski) localization.

(1,2) We shall prove that $L_{\A^1} B_\et^n G \wequi *$ (where $n=1$ unless $G$ is abelian); here $L_{\A^1}$ denotes the localization at $\A^1$-homotopy equivalences, also known as the \emph{singular construction} \cite[\S 2.3, \S 3.2]{A1-homotopy-theory}.
Suppose that $C \subset G$ is a central subgroup, and consider the exact sequence (of finite étale group schemes, equivalently étale sheaves) $1 \to C \to G \to Q \to 1$.
We claim that $(*)$ it suffices to prove the result for $C$ and $Q$.
Indeed since $C$ is central we have a fiber sequence of sheaves of spaces\todo{ref} $B_\et G \to B_\et Q \to B^2_\et C$ (respectively $B^n_\et G \to B^n_\et Q \to B^{n+1}_\et C$ if $G$ is abelian) which is preserved by $L_{\A^1}$ by \cite[Lemma 5.5.6.17]{HA}, which applies because $B^{n+1}_\et C$ is sectionwise connected \cite[Théorème X.5.1]{sga4}.
Since non-trivial $p$-groups have non-trivial centers \cite[Theorem I.6.5]{lang2002algebra}, finiteness of $G$ and quasi-compactness of $S$ allow an induction using $(*)$ on the size of $G$, which allows us to reduce to the case where $G$ is abelian.
By a similar induction (involving the kernel of multiplication by $p$ on $G$), we may even assume that $G$ is elementary abelian (i.e. locally isomorphic to $(\Z/p)^d$ for some $d$).

Since $G$ is finite étale and $S$ is connected, there exists a finite Galois cover $S' \to S$ over which $G$ becomes constant \cite[Tag 03SF]{stacks-project}, say isomorphic to $(\Z/p)^d$.
Let $S'/S$ have group of automorphisms $K$, another finite group, and let $P \subset K$ be a $p$-Sylow subgroup.
Then $S' \to S'/P$ is Galois with group $P$, and $q: S'/P \to S$ is finite étale of degree coprime to $p$.
Suppose that $L_{\A^1} B^n_\et q^*G \wequi *$.
The transfer in étale cohomology provides a morphism $q_*q^*B^n_\et G \to B^n_\et G$ (e.g. see \cite[Corollary C.13]{norms}), and the composite $B^n_\et G \to q_*q^* B^n_\et G \to B^n_\et G$ of restriction and transfer is an equivalence of étale sheaves (namely locally given by multiplication by the degree of $S'/P \to S$, which is coprime to $p$ by construction).
Since $L_{\A^1}$ commutes with $q_*$, and retracts of final objects are final, it suffices to prove the result over $S'/P$; in other words we may assume that $G$ is split by a Galois cover with group $P$.
The étale group scheme $G$ is determined completely by the action of $P$ on $(\Z/p)^d$, i.e. a $d$-dimensional representation of $P$ over $\F_p$.
Since non-trivial representations of $p$-groups over $\F_p$ are reducible \cite[\S8.3, Corollary of Proposition 26]{serre1977linear}, $G$ can be obtained as an iterated extension of copies of $\Z/p$ with the trivial action, i.e. the finite discrete group $C_p$.

We are thus (using $(*)$ again) reduced to proving that $L_{\A^1} B^n_\et C_p \wequi *$ for $n \ge 1$.
This is an easy consequence of the Artin--Schreier sequence; see \cite[proof of Lemma 5.5]{BEHKSY}.

(3) Let $S'/S$ be the scheme of Sylow $p$-subgroups of $G$.\todo{ref?---also for existence of sylow after pullback}
This is a finite étale scheme of degree the number of Sylow $p$-subgroups (of a geometric fiber of $G$), so in particular coprime to $p$ \cite[Theorem I.6.4(iii)]{lang2002algebra}.
By Lemma \ref{lem:coprime-bc-stuff}(1) below, base change along $S' \to S$ is conservative on $\SH(\ph)_{(p)}$, and hence we may replace $S$ by $S'$.
By construction, $G$ now admits a Sylow $p$-subgroup $P$ (i.e. a $p$-subgroup scheme of index coprime to $p$).
Let $E_{gm} G$ be an ind-smooth free contractible $G$-scheme, so that $B_\et G \wequi (E_{gm}G)/G$ and $B_\et P \wequi (E_{gm} G)/P$ \cite[\S4.2]{A1-homotopy-theory}.
Since $(E_{gm} G)/P \to (E_{gm} G)/G$ is a finite étale morphism, we can use the extension of $\SH(\ph)$ to ind-schemes \cite[\S A]{kleen2018motivic} to construct a transfer map $\Sigma^\infty_+ B_\et G \to \Sigma^\infty_+ B_\et P$ \cite[Definition 2.4, Theorem 2.8, Lemma A.3]{kleen2018motivic}.
It follows from Lemma \ref{lem:coprime-bc-stuff}(2) below and the formal definition of $\SH((E_{gm}G)/G)$ as a limit that the composition of transfer and restriction becomes an equivalence in $\SH(S)_{(p)}$.
By (1), $\Sigma^\infty_+(B_\et P)_{(p)} \wequi \1_{(p)}$, and hence $\Sigma^\infty_+(B_\et G)_{(p)}$ is a summand thereof.
But since $B_\et G$ admits a base point, $\Sigma^\infty_+(B_\et G)_{(p)}$ already splits off $\1_{(p)}$ as a summand, and by what we did the complementary summand is zero.
This concludes the proof.
\end{proof}

Recall that if $X \in \Sm_S$ is such that $\Sigma^\infty_+ X \in \SH(S)$ is strongly dualizable (e.g. $X$ smooth proper \cite[Proposition 2.4.31 and Theorem 2.4.28]{triangulated-mixed-motives}), then one has the \emph{transfer} $\1 \to \Sigma^\infty_+ X \in \SH(S)$; see e.g. \cite[Definition 2.3]{kleen2018motivic}.
Composing the transfer with (the image of) the structure map $X \to S$ we obtain an endomorphism of $\1 \in \SH(S)$ which is called the \emph{Euler characteristic of $X$}.
\begin{lem}\label{lem:coprime-bc-stuff}
Let $S$ be a scheme and $p$ a prime.
Suppose that $p=2$ or no residue field of $S$ is formally real (e.g. $S$ of positive characteristic).
Let $q: S' \to S$ be a finite étale morphism of degree coprime to $p$
\begin{enumerate}
\item The base change $q^*: \SH(S)_{(p)} \to \SH(S')_{(p)}$ is conservative.
\item The Euler characteristic of $S'$ becomes invertible in $\SH(S)_{(p)}$.
\end{enumerate}
\end{lem}
\begin{proof}
(2) The problem is Zariski local on $S$, whence we may assume $S$ affine.
We can write $S = \lim_i S_i$ for schemes $S_i$ of finite type over $\Z$ (without formally real residue fields), and then $S'$ is defined (and finite étale of degree coprime to $p$) already over some $S_i$.
We may thus assume that $S$ has finite dimension.
By \cite[Theorem B.3]{norms} we may further assume that $S$ is the spectrum of a field $k$.

The Euler characteristic is now an element of $GW(k)$ of rank equal to the degree of $q$ (this can be checked after geometric base change, where it is easily verified by direct computation).
It is thus enough to show that $GW(k)_{(p)} \to \Z_{(p)}$ detects units.
If $p$ is odd then the map is an isomorphism by our assumption on residue fields \cite[Theorem III.3.6]{milnor1973symmetric}.
If $p=2$, it suffices to prove that $GW(k)_{(2)}$ is a local ring with residue field $\Z/2$, for which see e.g. \cite[Lemma 2.13]{bachmann-eta}.

(1) For $E \in \SH(S)_{(p)}$ we get $q_\sharp q^* E \wequi E \otimes \Sigma^\infty_+(S')_{(p)}$, which contains $E$ as a summand by (2).
Hence if $q^*E=0$ also $E=0$, as desired.
\end{proof}

\begin{rmk} \label{rmk:generalized-contractibility}
Let $\scr C: \Sch_{\F_p} \to \Cat_\infty$ be a functor and $F: \Spc(\ph) \to \scr C$ be a natural transformation.
Analyzing the proofs in this section, we see that $F(B_\et G) \wequi 0$ for every $G$ and every $S \in \Sch_{\F_p}$ as soon as finite étale schemes become strongly dualizable in $\scr C$, $\scr C$ is additive, satisfies continuity and Zariski descent, and satisfies the analog of Lemma \ref{lem:coprime-bc-stuff}.
The latter will be true if base change to fields is conservative for $\scr C$ on finite type $\Z$-schemes (which itself is a consequence of continuity and localization \cite[Corollary 14]{bachmann-real-etale}) and Euler characteristic of finite étale $k$-algebras of degree coprime to $p$ are invertible in $[\1, \1]_{\scr C(k)}$ (which will hold if this group is $GW(k)_{(p)}$ and Euler characteristics are compatible with ranks).
\end{rmk}

\section{Enriched $\infty$-categories} \label{sec:enriched}
Let $\scr V$ be an $\infty$-category with finite products.
In \cite{gepner2015enriched}, the authors construct a theory of $\scr V$-enriched $\infty$-categories.\footnote{In fact they define this for any monoidal $\infty$-category $\scr V$, but we will not need to additional generality.}
The $\scr V$-enriched $\infty$-categories form an $\infty$-category $\Cat^{\scr V}_\infty$.
If $\scr C$ is a $\scr V$-enriched $\infty$-category, then $\scr C$ has a space of objects $\scr C^\wequi \in \Spc$, and for every pair of objects $X, Y \in \scr C$ there is a mapping object $\ul\Map_\scr{C}(X, Y) \in \scr V$.
There are coherently associative composition maps for these $\scr V$-valued mapping spaces.

The main point for us will be that $\Cat^\scr{V}_\infty$ is functorial in $\scr V$ for functors preserving finite products \cite[Corollary 5.7.6]{gepner2015enriched}.
That is, given a functor $F: \scr V \to \scr V'$ preserving finite products and a $\scr V$-enriched $\infty$-category $\scr C$, one may form a $\scr V'$-enriched $\infty$-category $F\scr C$, with space of objects $(F\scr C)^\wequi$ receiving a surjection from $\scr C^\wequi$ and $\scr V'$-valued mapping spaces $\ul\Map_{F\scr C}(X,Y) \wequi F\ul\Map_{\scr C}(X',Y')$, where $X', Y'$ are lifts of $X,Y$; the composition maps in $F\scr C$ are induced from those in $\scr C$ using that $F$ preserves finite products.
\begin{exm}
We always have a functor $\scr V \to \Spc$, $V \mapsto \Map(*, V)$.
It preserves finite products and so produces for every $\scr V$-enriched $\infty$-category $\scr C$ an $\Spc$-enriched $\infty$-category, i.e. an $\infty$-category in the usual sense.
This is the \emph{underlying $\infty$-category}.
It has the same space of objects as $\scr C$.
\end{exm}
\begin{rmk} \label{rmk:enriched-equiv}
A functor $f: \scr C \to \scr C' \in \Cat^\scr{V}_\infty$ is called \emph{fully faithful} if it induces equivalences $\ul\Map_{\scr C}(X,Y) \wequi \ul\Map_{\scr C'}(fX, fY)$ for all $X,Y \in \scr C$.
It is called \emph{essentially surjective} if the underlying functor of ordinary $\infty$-categories is essentially surjective (i.e. the induced map on spaces of objects is surjective).
Essentially by definition, $f$ is an equivalence if and only if it is fully faithful and essentially surjective.
\end{rmk}

Functoriality of $\Cat^\scr{V}_\infty$ in $\scr V$ already means that given any localization $L: \scr V \to \scr V$ preserving finite products, we can localize the mapping spaces functorially on $\scr V$-enriched $\infty$-categories.
We do not strictly need this, but observe that this localization of mapping spaces even satisfies a universal property.
\begin{lem} \label{lem:mapping-space-loc}
Let $\scr V$ be a presentable $\infty$-category and $L: \scr V \to \scr V$ a (Bousfield) localization preserving finite products.
Denote by $L\Cat^\scr{V}_\infty$ the full subcategory on those $\scr V$-enriched $\infty$-categories where the $\scr V$-valued mapping spaces are $L$-local.
Then the inclusion $L\Cat^\scr{V}_\infty \to \Cat^\scr{V}_\infty$ is right adjoint to the functor $L: \Cat^\scr{V}_\infty \to L\Cat^\scr{V}_\infty$.
\end{lem}
\begin{proof}
Write $L': V \adj LV: \iota$ for the adjunction.
By \cite[Proposition 5.7.17]{gepner2015enriched}, there is an induced adjunction $L'_*: \Cat^\scr{V}_\infty \adj \Cat^{L\scr V}_\infty: \iota_*$.
It follows from the explicit description of $L'_*$ and $\iota_*$ that for $\scr C \in \Cat^{L\scr V}_\infty$ the co-unit $L'_* \iota_* \scr C \to \scr C$ is essentially surjective and fully faithful, whence an equivalence by Remark \ref{rmk:enriched-equiv}.
It follows that $L'_*$ is a Bousfield localization, and again the explicit description of $L'_*, \iota_*$ shows that the local objects are precisely the categories with $L$-local $\scr V$-valued mapping spaces.
\end{proof}

We now exhibit a convenient way of constructing $\infty$-categories enriched in presheaves.
Let $\scr W$ be a small $\infty$-category with a final object $*$.
We shall show that given $\scr C: \scr W^\op \to \Cat_\infty$ we can form a $\PSh(\scr W)$-enriched $\infty$-category $\ul{\scr C}$ with underlying $\infty$-category $\scr C(*)$ and mapping presheaves given by \begin{equation} \label{eq:RC-maps} \PSh(\scr W) \ni \ul\Map_{\ul{\scr C}}(X,Y): w \mapsto \Map_{\scr C(w)}(X_w,Y_w). \end{equation}

Since $\scr W$ has a final object, the full subcategory of $\PSh(\scr W)$ on constant presheaves is equivalent to $\Spc$ (via evaluation at $*$).\NB{One argument: $\PSh_{cnst}(\scr W) \wequi \PSh(|\scr W|) \wequi \PSh(*) \wequi \Spc$}
\def\Seg{\mathrm{Seg}}
\def\cat{\mathrm{cat}}
Denote by $\Seg(\PSh(W)) \subset \PSh(\Delta \times \scr W)$ the full subcategory on those objects whose restriction to $\Delta \times \{c\}$ is a Segal space for every $c \in \scr W$, and whose restriction to $\{\Delta^0\} \times \scr W$ is constant.

We have so far avoided giving an actual definition of $\Cat_\infty^{\scr V}$.
For the next result, we need to get a little bit more technical: $\Cat_\infty^{\scr V}$ is in fact obtained as a localization of $\Alg_\cat(\scr V)$, the category of ``categorical algebras in $\scr V$'' \cite[Definition 4.3.1]{gepner2015enriched}.
\begin{prop} \label{prop:PSh-enriched}
Let $\scr W$ be a small $\infty$-category with a final object $*$.
There is an equivalence \[ \Alg_\cat(\PSh(\scr W)) \wequi \Seg(\PSh(\scr W)) \] such that the full subcategory $\Cat^{\PSh(\scr W)}_\infty \subset \Alg_\cat(\PSh(\scr W))$ corresponds to the full subcategory of those objects $E_\bullet \in \Seg(\PSh(\scr W))$ such that $E_\bullet(*) \in \Seg(\Spc)$ is complete.
\end{prop}
\begin{proof}
This is a special case of \cite[Definition 7.4, Theorem 7.5, Definition 7.13 and Proposition 7.16]{haugseng2015rectification} once we observe that $\PSh(W)$ is an \emph{absolute distributor} \cite[Definition 7.2]{haugseng2015rectification}.
This follows from the fact that $\PSh(W)$ is an $\infty$-topos (and $W$ is weakly contractible), using \cite[Theorems 6.1.0.6 and 6.1.3.9]{HTT}.
\end{proof}

\begin{cnstr} \label{cnstr:enriched}
Let $\scr W$ be a small $\infty$-category with a final object $*$.
Write $e: \Delta \wequi * \times \Delta \to \scr W \times \Delta$ and $f: \scr W \wequi \scr W \times \{\Delta^0\} \to \scr W \times \Delta$ for the inclusions, $e^*, f^*$ for the restriction functors induced on presheaf categories, and $e_!, f_*$ for the left/right adjoints.
Consider the functor \[ R: \Fun(\scr W^\op, \Cat_\infty) \hookrightarrow \PSh(\scr W \times \Delta) \xrightarrow{R'} \PSh(\scr W \times \Delta), \] where \[ R'F = e_!e^*f_*f^*F \times_{f_*f^* F} F. \]
\end{cnstr}

\begin{rmk} \label{rmk:explain-R}
Write $c: \Seg(\Spc) \to \Seg(\Spc)$ for the composite $g_*g^*$, where $g: \{\Delta^0\} \to \Delta$.
Then $c$ is the functor making the mapping spaces in a Segal space discrete, while retaining the space of objects.
One verifies easily that $(f_*f^*F)(w) \wequi c(F(w))$, and $e_!e^*f_*f^*F$ is the constant presheaf (of Segal spaces) on $c(F(*))$.
Thus informally speaking, the functor $R$ takes a presheaf of categories (i.e. complete Segal spaces) $F$ and constructs a presheaf of Segal spaces $RF$, all of which have the same space of objects (namely $F(*)^\wequi$), but the mapping spaces in $(RF)(w)$ are those of $F(w)$.
\end{rmk}

\begin{cor} \label{cor:PSh-enriched}
Let $\scr W$ be a small $\infty$-category with a final object $*$ and $\scr C \in \Fun(W^\op, \Cat_\infty)$.
Then \[ R\scr C \in \Cat^{\PSh(\scr W)}_\infty \subset \Alg_\cat(\PSh(\scr W)) \wequi \Seg(\PSh(\scr W)) \subset \PSh(\scr W \times \Delta). \]
The objects of $R\scr C$ are the objects of $\scr C(*)$, and given $X, Y \in \scr C(*)$ the mapping presheaf $\ul\Map_{R\scr C}(X, Y) \in \PSh(\scr W)$ is given by \eqref{eq:RC-maps}.
\end{cor}
\begin{proof}
Let $F \in \Fun(\scr W^\op, \Seg(\Spc)) \subset \PSh(\scr W \times \Delta)$.
Remark \ref{rmk:explain-R} explains why $RF$ is a presheaf of Segal spaces, and what the spaces of objects and mapping spaces are.
The first claim follows from this since the Remark shows that $RF_0$ is constant and $RF(*) \wequi F(*)$ is complete.
The second claim is just a reformulation of Remark \ref{rmk:explain-R}.
\end{proof}

Now that we know $R\scr C$ is an $\infty$-category enriched in $\PSh(\scr W)$, we also denote it by $\ul{\scr C}$.
\begin{exm}
Let $\scr C \in \Cat_\infty$ and denote by $e\scr C \in \Fun(\scr W^\op, \Cat_\infty)$ the associated constant presheaf.
We obtain $\ul{e\scr C} \in \Cat^{\PSh(\scr W)}_\infty$, which has underlying $\infty$-category $\scr C$ and mapping presheaves the constant presheaves associated with the mapping spaces in $\scr C$.
\end{exm}

\def\surj{\mathrm{surj}}
\begin{rmk}
Denote by $\Fun(\scr W^\op, \Cat_\infty)_\surj \subset \Fun(\scr W^\op, \Cat_\infty)$ the full subcategory on those presheaves of categories all of whose transition maps are essentially surjective.
One may show that the restriction of $R$ to $\Fun(\scr W^\op, \Cat_\infty)_\surj$ is fully faithful and admits a left adjoint given by levelwise Segal completion.
From this one may deduce that $R$ induces an equivalence \[ \Fun(\scr W^\op, \Cat_\infty)_\surj \wequi \Cat^{\PSh(\scr W)}_\infty. \]
We will not need this more precise statement.
\end{rmk}

\section{Semiadditivity and ambidexterity} \label{sec:semiadd}
\subsection{Review of semiadditivity}
Recall that an $\infty$-category $\scr C$ is called \emph{semiadditive} if it is pointed (i.e. admits an object which is both initial and final) and admits finite biproducts (equivalently admits binary coproducts, and for $X, Y \in \scr C$ the span $X \amalg 0 \leftarrow X \amalg Y \to 0 \amalg Y$ is a limit diagram).
The $\infty$-category $\Cat_\infty^\amalg$ of (small) $\infty$-categories with finite coproducts, and functors preserving the finite coproducts, is presentable.
It has a canonical symmetric monoidal structure such that the left adjoint of the forgetful functor $\Cat_\infty^\amalg \to \Cat_\infty$ is symmetric monoidal \cite[Corollary 4.8.4.1]{HA}.
Let $\Cat_\infty^\oplus \subset \Cat_\infty^\amalg$ be the full subcategory on the semiadditive $\infty$-categories.
By \cite[\S5.1]{harpaz2020ambidexterity} we know that $\Span(\Fin) \in \Cat_\infty^\amalg$ is an idempotent monoid, and $\Cat_\infty^\oplus$ is its category of modules.
This immediately implies the following.
\begin{prop}[Harpaz] \label{prop:semiadd-basics}
The $\infty$-category $\Cat_\infty^\oplus$ of small, semiadditive $\infty$-categories (and additive functors) is presentable and admits a canonical symmetric monoidal structure such that the left adjoint of the forgetful functor $\Cat_\infty^\oplus \to \Cat_\infty$ is symmetric monoidal.
Moreover $\Span(\Fin)$ is the free semiadditive $\infty$-category on one generator (i.e. the result of applying the left adjoint to $* \in \Cat_\infty$).
\end{prop}
In particular if $\scr C$ is semiadditive and $X \in \scr C$, then there is a unique functor $\Span(\Fin) \to \scr C$ preserving finite coproducts and sending $*$ to $X$.
Similarly, if $\scr C$ has a symmetric monoidal structure preserving finite coproducts in each variable, then there is a unique symmetric monoidal functor $\Span(\Fin) \to \scr C$ preserving finite coproducts (indeed essentially by definition $\scr C$ promotes to an object of $\CAlg(\Cat_\infty^\oplus)$, and $\Span(\Fin)$ is the initial object of this category, being the unit of $\Cat_\infty^\oplus$).

\subsection{Motivic spaces with finite étale transfers}
One expects that a similar story plays out for motivic categories (i.e. appropriate presheaves of categories on $\Sch_S$), but developing this in generality would take us too far afield.
We contend ourselves with the following.
\begin{thm} \label{thm:Spcfet}
Denote by $\Span(\FEt_{(\ph)}) \in \Fun(\Sch^\op, \Cat_\infty)$ the functor $S \mapsto \Span(\FEt_S)$.
Then there is a canonical natural transformation \[ \Span(\FEt_{(\ph)}) \to \SH(\ph) \in \Fun(\Sch^\op, \Cat_\infty). \]
It sends $X \in \FEt_S$ to $\Sigma^\infty_+ X \in \SH(S)$.
\end{thm}
\begin{proof}
Combine \cite[Theorem 18]{hoyois2018localization} and \cite[\S4.3.15]{EHKSY}.
\end{proof}

\section{Main result} \label{sec:main}
Fix a scheme $S$.
The class of morphisms in $\PSh(\Sm_S)$ which become an equivalence after applying \[ \Sigma^\infty_+(\ph)_{(p)}: \PSh(\Sm_S) \to \SH(S)_{(p)} \] is strongly saturated and of small generation \cite[Proposition 5.5.4.16]{HTT}.
It follows \cite[Proposition 5.5.4.15]{HTT} that there exists a (Bousfield) localization \[ L: \PSh(\Sm_S) \to \PSh(\Sm_S) \] such that $L\alpha$ is an equivalence if and only if $\Sigma^\infty_+(\alpha)_{(p)}$ is an equivalence.
\begin{lem} \label{lem:L-basics}
Let $f: S' \to S$ be a morphism of schemes.
\begin{enumerate}
\item $f^*: \PSh(\Sm_S) \to \PSh(\Sm_{S'})$ preserves $L$-equivalences and $f_*: \PSh(\Sm_{S'}) \to \PSh(\Sm_S)$ preserves $L$-local objects.
\item If $f$ is smooth, then $f_\sharp: \PSh(\Sm_{S'}) \to \PSh(\Sm_S)$ preserves $L$-equivalences and $f^*$ preserves $L$-local objects.
\item If $f$ is finite étale, then $f_*$ preserves $L$-equivalences.
\end{enumerate}
In particular, $L$ preserves finite products.
\end{lem}
\begin{proof}
Whenever a left adjoint preserves $L$-equivalences, the right adjoint preserves $L$-local objects.
Since $f^*$ and $f_\sharp$ commute with $\Sigma^\infty_+(\ph)_{(p)}$, we get (1) and (2).
To prove (3), note that by \cite[Proposition 12.8]{norms}, $\1_{(p)} \in \SH(S)$ admits a normed structure.
It follows \cite[Proposition 7.6(4)]{norms} that there is a functor $q_\otimes$ making the following diagram commute
\begin{equation*}
\begin{CD}
\Spc(S') @>>> \SH(S')_{(p)} \\
@V{q_*}VV   @V{q_\otimes}VV \\
\Spc(S)  @>>> \SH(S)_{(p)}.
\end{CD}
\end{equation*}
The result follows.
\end{proof}

I have been unable to prove that $L$ preserves coproducts (of Nisnevich sheaves, say).
The following is a weak substitute.
\begin{lem} \label{lem:L-cnst}
Constant (Zariski) sheaves of sets in $\PSh(\Sm_S)$ are $L$-local.
\end{lem}
\begin{proof}\todo{What the hell?}
Note that for any set $A$, $L_\Sigma A$ is the sheaf of locally constant $A$-valued functions; i.e. this is the constant sheaf on $A$.
The empty constant sheaf is strictly initial in Zariski sheaves (say), and hence given any map $X \to Y \in \Shv_\Zar(\Sm_S)$ we have $\Map(X, L_\Sigma \emptyset) \wequi \Map(Y, L_\Sigma \emptyset)$ unless possibly $X = L_\Sigma \emptyset$ and $Y \ne L_\Sigma \emptyset$.
But then $\Sigma^\infty_+(X)_{(p)} = 0$ whereas $\Sigma^\infty_+(Y)_{(p)} \ne 0$ (indeed $Y$ admits a base point after pullback to some non-empty smooth $S$-scheme, and hence $\Sigma^\infty_+(Y)_{(p)}$ splits off a copy of $\1_{(p)}$ after such a base change\NB{and $\1_{(p)} \ne 0$ because the category is not zero, e.g. see construction below}).
Hence the empty constant sheaf is $L$-local.

Now let $B \hookrightarrow A$ be an injection of sets, where $B$ is non-empty.
Then the injection splits and so the constant sheaf on $B$ is a retract of the constant sheaf on $A$.
Since $L$-local objects are stable under retracts, it follows that $L_\Sigma B$ is $L$-local as soon as $L_\Sigma A$ is.

It suffices thus to find arbitrarily large sets $A$ which have $L$-local constant sheaves.
We will do so Zariski locally on $S$; since $L$-equivalences are Zariski (in fact Nisnevich) local, Zariski sheaves being $L$-local is a Zariski local property, and so this is enough.
We may thus assume that a prime $\ell \ne p$ is invertible on $S$, and we take $A = (\Z/\ell)^C$ for $C$ arbitrarily large.
Let $E \in \SH(S)_{(p)}$ be the spectrum representing étale motivic cohomology with $\Z/\ell(*)$-coefficients (this exists because étale cohomology is $\A^1$-invariant and satisfies a $\P^1$-bundle formula \cite[Corollaire XV.2.2]{sga4} \cite[Proposition VII.1.1(ii)]{sga5}).
Then $\Omega^\infty E \wequi a_\et \Z/\ell$ is the constant sheaf on $\Z/\ell$, and so \[ \Omega^\infty \prod_C E \wequi (\Z/\ell)^C \] as needed.
\end{proof}

Since $L$ preserves finite products by Lemma \ref{lem:L-basics}, it follows from Lemma \ref{lem:mapping-space-loc} that given a $\PSh(\Sm_S)$-enriched $\infty$-category, we may functorially $L$-localize all the mapping spaces.
Restricting the transformation of Theorem \ref{thm:Spcfet} to $\Sm_S$, composing with $\SH(\ph) \to \SH(\ph)_{(p)}$ and applying Construction \ref{cnstr:enriched}, we obtain a functor of $\PSh(\Sm_S)$-enriched categories \[ \ul{\Span(\FEt_S)} \to \ul{\SH(S)_{(p)}}. \]

\begin{lem} \label{lem:SHp-uMap-local}
The $\PSh(\Sm_S)$-valued mapping spaces in $\ul{\SH(S)_{(p)}}$ are $L$-local.
\end{lem}
\begin{proof}
For $E, F \in \SH(S)_{(p)}$, denote by $M \in \SH(S)_{(p)}$ the internal mapping object, i.e. the result of applying the right adjoint of $\otimes E$ to $F$.
Then for $X \in \Sm_S$ we have \[ \ul\Map(E,F)(X) = \Map(E_X, F_X) \wequi \Map(\Sigma^\infty_+ X \otimes E, F) \wequi \Map(\Sigma^\infty_+ X, M). \]
It follows\NB{... something about compatibility of enrichment and closed symmetric monoidal structure} that $\ul\Map(E,F) \wequi \Omega^\infty M$, which is $L$-local essentially by definition.
\end{proof}

\begin{lem} \label{lem:uMap-SpanFet}
Let $X, Y \in \Span(\FEt_S)$.
We have \[ \ul\Map(X,Y) \wequi q_*q^* L_\Sigma \coprod_{n \ge 0} B_\et \Sigma_n, \] where $q: X \times_S Y \to S$ is the projection.
\end{lem}
\begin{proof}
Since $Y$ is self-dual in $\Span(\FEt_S)$, we have $\ul\Map(X,Y) \wequi \ul\Map(X \times Y, S)$, and this latter presheaf coincides with $q_*q^* \ul\Map(S,S)$ essentially by definition.
Finally $\Map_{\Span(\FEt_S)}(S,S) \wequi \FEt_S^\wequi$ naturally in $S$, so that $\ul\Map(S,S) \wequi \FEt_{(\ph)}^\wequi$.
The result follows since any finite étale scheme over $S$ is isomorphic to one of constant degree, locally on a clopen cover of $S$ (and $B_\et \Sigma_n$ is the groupoid of finite étale schemes of degree $n$).
\end{proof}

Write $L\Span(\FEt_S)$ for the ordinary $\infty$-category underlying $L\ul{\Span(\FEt_S)}$.
Since $\Span(\FEt_S)$ is semiadditive, there is a unique functor $\Span(\Fin) \to \Span(\FEt_S)$ sending $*$ to $S$.
\begin{thm} \label{thm:main-idea}
Let $S$ be an $\F_p$-scheme.
\begin{enumerate}
\item The $\infty$-category $L\Span(\FEt_S)$ is a $1$-category.
  In fact for $X, Y \in \FEt_S$ we have \[ \ul\Map_{L\ul{\Span(\FEt_S)}}(X, Y) \wequi q_* L_\Sigma \NN, \] where $q: X \times_S Y \to S$ is the projection.
\item If $S$ is connected, then the functor $\h\Span(\Fin) \to L\Span(\FEt_S)$ (induced via (1) by the composite $\Span(\Fin) \to \Span(\FEt_S) \to L\Span(\FEt_S)$) is fully faithful.
\end{enumerate}
\end{thm}
\begin{proof}
(1) It suffices to prove the second statement.
Using Lemmas \ref{lem:uMap-SpanFet} and \ref{lem:L-basics}, it suffices to show that $L B_\et \Sigma_n \wequi *$ and that $L_\Sigma \NN$ is $L$-local.
The first statement is Proposition \ref{prop:Bet}, and the second is Lemma \ref{lem:L-cnst}.

(2) Using semiadditivity, it suffices to show that $\pi_0 \Map_{\Span(\Fin)}(*,*) \wequi \Map_{L\Span(\FEt_S)}(S,S)$.
Since $S$ is connected, both sides canonically identify with $\NN$ (and so each other).
\end{proof}

Recall that for a commutative (semi)ring $R$, we write $\FFree_R$ for the category of finitely generated free $R$-modules.
\begin{cor} \label{cor:main-result}
Let $S$ be an $\F_p$-scheme.
The functor $\Span(\Fin) \to \SH(S)_{(p)}$ factors canonically (as a symmetric monoidal functor) through $\Span(\Fin) \to \h\Span(\Fin) \wequi \FFree_\NN \to \FFree_\Z \to \FFree_{\Z_{(p)}}$.
\end{cor}
\begin{proof}
Since $\Span(\FEt_S)$ is semiadditive and $\Span(\FEt_S) \to \SH(S)_{(p)}$ preserves finite coproducts, it follows from Proposition \ref{prop:semiadd-basics} that $\Span(\Fin) \to \SH(S)_{(p)}$ factors canonically through $\Span(\FEt_S)$.
By Lemma \ref{lem:SHp-uMap-local} we have $L\ul{\SH(S)_{(p)}} \wequi \ul{\SH(S)_{(p)}}$.
Thus we obtain a symmetric monoidal factorization \[ \Span(\Fin) \to \Span(\FEt_{S}) \to L\Span(\FEt_{S}) \to L\SH(S)_{(p)} \wequi \SH(S)_{(p)}, \] where $L\SH(S)_{(p)}$ denotes the $\infty$-category underlying $L\ul{\SH(S)_{(p)}}$.
(To see that this factorization is symmetric monodial it suffices to observe that $\ul{\Span(\FEt_S)} \to \ul{\SH(S)_{(p)}}$ is a morphism in $\CMon(\Cat^{\PSh(\Sm_S)}_\infty)$ and $L: \Cat^{\PSh(\Sm_S)}_\infty \to \Cat^{\PSh(\Sm_S)}_\infty$ preserves finite products.)
By Theorem \ref{thm:main-idea}, $L\Span(\FEt_{S})$ is a $1$-category and so $\Span(\Fin) \to L\Span(\FEt_{S})$ factors through $\h\Span(\Fin) \wequi \FFree_\NN$.
The induced functor $\PSh_\Sigma(\FFree_\NN) \to \SH(S)_{(p)}$ has a right adjoint landing in grouplike strictly commutative monoids, and hence factors through $\PSh_\Sigma(\FFree_\NN)^\gp$.
Since $\NN^\gp \wequi \Z$ (this follows from \cite[Proposition 1]{mcduff1976homology}), we obtain the factorization through $\PSh_\Sigma(\FFree_\Z)$ (and so $\FFree_\Z$), which is the category of connective $H\Z$-modules.
Finally since $\SH(\F_{p})_{(p)}$ is $p$-local, we get a further factorization through the $p$-localization of $H\Z$-modules, i.e. $H\Z_{(p)}$-modules, and so through $\FFree_{\Z_{(p)}}$.
\end{proof}

\begin{rmk} \label{rmk:alt-arg}
An alternative argument, avoiding Lemma \ref{lem:L-cnst}, proceeds as follows.
Denote by $L_\et^0: \PSh(\Sm_S) \to \PSh(\Sm_S)$ the localization $F \mapsto a_\et \pi_0 F$.
Then for $X,Y,q$ as in Lemma \ref{lem:uMap-SpanFet} we have $q_* L_\Sigma \NN \wequi L_\et^0 \ul\Map(X,Y)$ (using only Lemma \ref{lem:uMap-SpanFet}).
We can thus form the enriched $\infty$-category $L_\et^0 \ul{\Span(\FEt_S)}$ with mapping presheaves of the form $q_* L_\Sigma \NN$.
Proposition \ref{prop:Bet} shows that for any $\F_p$-scheme $S$, the functor $\ul{\Span(\FEt_S)} \to L_\et^0 \ul{\Span(\FEt_S)}$ is an $L$-equivalence.
Since $\ul{\SH(S)_{(p)}}$ is $L$-local, $\ul{\Span(\FEt_S)} \to \ul{\SH(S)_{(p)}}$ factors through $L_\et^0 \ul{\Span(\FEt_S)}$ (by Lemma \ref{lem:mapping-space-loc}).
Since the latter is a $1$-category, it follows that $\Span(\Fin) \to \SH(S)_{(p)}$ factors through $\h\Span(\Fin)$, as desired.
\end{rmk}

\section{Complements} \label{sec:complements}
\addtocontents{toc}{\protect\setcounter{tocdepth}{2}}
In this section we establish some complementary results.
In \S\ref{subsec:SHS-fet} we recall the category $\SH^{S^1,\fet}(S)$ of motivic $S^1$-spectra with finite étale transfers and we prove that in the formulation of the main theorem, the category $\SH(S)_{(p)}$ can be replaced by $\SH^{S^1,\fet}(S)_{(p)}$, provided that $p \ne 2$.
We use this in \S\ref{subsec:picard} to show that there is a map $H\Z \to \Pic(\SH(\F_p)_{(p)})$ classifying $\Sigma^{2,1} \1_{(p)}$.
Then in \S\ref{subsec:graded-mult} we generalize the descriptions of $\CMon(\Spc)$ and $\Spc^\fet(S)$ of \S\ref{sec:semiadd} to the case of commutative (normed) rings, graded on some commutative monoid $A$.
We use this in \S\ref{subsec:normed-SCR} to prove a variant of our main result: if $E \in \NAlg(\SH(\F_p)_{(p)}^A)$ then the global sections of $E$ has the structure of an $A$-graded simplicial commutative ring.
Finally in \S\ref{subsec:MGL} we show that the sphere graded global sections of $\MGL_{(p)}$ over $\F_p$ carry the structure of a $\Z$-graded simplicial commutative ring, and admit a ring map from the truncated $\Z$-graded ring $L_*$ (the Lazard ring).

\subsection{Variant for $S^1$-spectra with finite étale transfers} \label{subsec:SHS-fet}
We put $\SH^{S^1,\fet}(S) = \Spc^\fet(S)[(S^1)^{-1}]$ (see \cite[\S3.3]{bachmann-MGM} for basic facts about the category $\Spc^\fet(S) = L_\mot \PSh_\Sigma(\Cor^\fet(\Sm_S))$ of motivic spaces with finite étale transfers).
Here are some properties of this construction.

\begin{prop} \label{prop:S1-fet} \hfill
\begin{enumerate}
\item If $f: S' \to S$ is finite étale, then $\Sigma^\infty_{S^1} f_\otimes(S^1) \in \SH^{S^1,\fet}(S)$ is invertible.
\item $\SH^{S^1,\fet}(\ph)$ has norms compatible with those on $\Spc(\ph)_*$ (and $\Spc^\fet(\ph)$).
\item $\SH^{S^1,\fet}(\ph)$ satisfies continuity and localization (in the sense of \cite[\S A.5.3, \S2.3]{triangulated-mixed-motives}), and is a Nisnevich sheaf.
\item The adjunction $\Spc^\fet(S) \adj \SH^{S^1,\fet}(S)$ induces an equivalence between objects of $\Spc^\fet(S)$ which are grouplike and strongly $\A^1$-invariant (in the sense of \cite[Definition 3.1.6]{EHKSY}) and objects of $\SH^{S^1,\fet}(S)$ which are connective as Nisnevich sheaves of spectra.
\item If $k$ is a field of characteristic $\ne 2$ we have $\ul{\pi}_0(\1_{\SH^{S^1,\fet}(k)}) \wequi \ul{GW}$.
\end{enumerate}
\end{prop}
\begin{proof}
(1) The inclusion $\Fin_{\widehat\Pi_1^\et(S)} \wequi \FEt_S \to \SmQP_S$ preserves norms (see \cite[\S10.2]{norms} for details) and hence induces a functor \[ \PSh_\Sigma(\Span(\Fin_{\widehat\Pi_1^\et(S)}))[(S^1)^{-1}] \to \SH^{S^1,\fet}(S). \]
It is thus enough to prove that $f_\otimes(S^1)$ is invertible in the source category.
This follows from \cite[Proposition 9.11]{norms} (since $f_\otimes(S^1)$ is invertible in the category $\SH(X)$ of \emph{loc. cit.}, by definition).

(2) Immediate consequence of (1).

(3) Continuity is standard.
For localization see \cite[Remark 12]{hoyois2018localization}.
For descent see \cite[Lemma 6.1]{bachmann-MGM}.

(4) The proof of \cite[Proposition 3.1.9]{EHKSY} goes through essentially unchanged.

(5) Since $\PSh_\Sigma(\Cor^\fet(\Sm_S)^\op) \to \PSh_\Sigma(\Sm_{S+})$ commutes with $L_\mot$ \cite[Lemma 3.26]{bachmann-MGM}, the forgetful functor $\Fun^\times(\Cor^\fet(\Sm_S)^\op, \SH) \to \Fun^\times(\Sm_{S}^\op, \SH)$ also does (since the forgetful functor on prespectra commutes with levelwise motivic localization and spectrification).
It follows that $\ul{\pi}_0(\1_{\SH^{S^1,\fet}(k)})$ is the initial strictly $\A^1$-invariant sheaf of abelian groups under $\pi_0 (\FEt^{\wequi})^\gp$ (using that $\1_{\PSh_\Sigma(\Cor^\fet(S))} \wequi \FEt^\wequi$).
This is $\ul{GW}$ by \cite[Proposition 7]{bachmann-claim}.
\end{proof}

We can now replace $\SH(\ph)$ by $\SH^{S^1,\fet}(\ph)$ in the main theorem.
\begin{prop} \label{prop:main-variant}
Let $S$ be an $\F_p$-scheme, where $p \ne 2$.
The canonical functor $\Span(\Fin) \to \SH^{S^1,\fet}(\ph)_{(p)}$ factors through $\h\Span(\Fin)$.
\end{prop}
\begin{proof}
The functor $\Spc(S) \to \SH^{S^1,\fet}(S)_{(p)}$ contracts $B_\et G$ for any $G$.
Indeed the requirements of Remark \ref{rmk:generalized-contractibility} hold, by Proposition \ref{prop:S1-fet}(3,5) (note also that finite étale schemes are dualizable in $\Cor^\fet(S)$, essentially by construction).
Define a localization $L': \Spc(S) \to \Spc(S)$ with equivalences detected by $\Spc(S) \to \SH^{S^1,\fet}(S)_{(p)}$.
It follows that $L'$ satisfies all the properties of the localization $L$ used in \S\ref{sec:main}, and so the arguments generalize to give the desired result.
\end{proof}

\begin{rmk}
Since constant sheaves are strictly $\A^1$-invariant, the functor $\SH \to \SH^{S^1}(\bar\F_p)$ is fully faithful.
It follows that mapping spaces in $\SH^{S^1}(\bar\F_p)_{(p)}$ are \emph{not} strictly commutative, i.e. it is not possible to replace $\SH(\ph)$ by $\SH^{S^1}(\ph)$ in the main result.
\end{rmk}

\subsection{Strictly commutative Picard spectra} \label{subsec:picard}
If $\scr C$ is a symmetric monoidal $\infty$-category, we denote by $\Pic(\scr C) \subset \scr C^\wequi$ the full subgroupoid on $\otimes$-invertible objects.
This defines a grouplike object of $\CMon(\Spc)$, i.e. a connective spectrum.

Denote by $\ul\Pic^0(\SH(S)_{(p)}) \in \PSh(\Sm_S)$ the presheaf of subgroupoids of $\SH(\ph)_{(p)}^\wequi$ on objects which are Nisnevich locally equivalent to $\Sigma^{2n,n}\1_{(p)}$ for some $n$.
\begin{lem} \label{lem:pic-sp}
$\ul\Pic^0(\SH(S)_{(p)})$ upgrades to an object of $L_\Nis \Fun^\times(\Cor^\fet(S)^\op, \SH)$.
If $S$ is the spectrum of a field and $p$ is odd\NB{Not minimal assumptions.}, this object is motivically local, i.e. defines an object of $\SH^{S^1,\fet}(S)$.
\end{lem}
\begin{proof}
Norms on $\SH(\ph)_{(p)}$ lift the (Nisnevich local) presheaf $\SH(\ph)_{(p)}^\wequi$ to $L_\Nis \PSh_\Sigma(\Cor^\fet(S))$.
Since $L_\Nis \PSh_\Sigma(\Cor^\fet(S))^\gp \wequi L_\Nis \Fun^\times(\Cor^\fet(S)^\op, \SH_{\ge 0})$ (combine \cite[Corollary 2.5]{gepner2016universality} with $\CMon(\Spc)^\gp \wequi \SH_{\ge 0}$), for the first assertion it suffices to prove that if $E \in \SH(S')_{(p)}$ is locally equivalent to a sphere and $f: S' \to S$ is finite étale, then also $f_\otimes(E)$ is locally equivalent to a sphere.
Using that any Nisnevich cover of $S'$ can, up to clopen refinement, be obtained by pulling back a Nisnevich cover of $S$ (use \cite[Tag 04GH(1)]{stacks-project}), we reduce to the case where $E$ itself is a sphere.
This case is handled via \cite[Lemma 4.4]{norms}.

For the second assertion we need only prove that each $\ul\pi_i := \ul\pi_i \ul\Pic^0(\SH(k)_{(p)})$ is strictly $\A^1$-invariant \cite[Theorem 6.2.7]{morel2005stable}.
This is clear for $\ul\pi_0 = \Z$ and $\ul\pi_i \wequi \ul\pi_{i-1}(\1_{(p)})$ for $i>1$.
On the other hand $\ul\pi_1 = (\ul{GW}_{(p)})^\times$.
We denote by $a_\ret$ sheafification in the real étale topology (see \cite{real-and-etale-cohomology}).
Since $p$ is odd and $\ul{GW}[1/2] \wequi a_\ret \Z[1/2] \times \Z[1/2]$ \cite[Lemma 39, Corollary 19]{bachmann-real-etale} we have $(\ul{GW}_{(p)})^\times \wequi (a_\ret \Z_{(p)})^\times \times \Z_{(p)}^\times$, which is strictly $\A^1$-invariant \cite[Example 16.7.2]{real-and-etale-cohomology}.
\end{proof}

\begin{cor} \label{cor:HZ-Pic}
Let $S$ be an $\F_p$-scheme, where $p \ne 2$.
There is a morphism of connective spectra $H\Z \to \Pic(\SH(S)_{(p)})$ sending $1 \in \Z$ to $\Sigma^{2,1} \1_S$.
\end{cor}
\begin{proof}
We may assume that $S = \Spec(\F_p)$.
Let $\scr P \subset \Pic(\SH(\F_p)_{(p)})$ be the subgroupoid on the spheres $\Sigma^{2n,n}\1$, so that $\scr P$ is also the global sections of $\ul\Pic^0(\SH(\F_p)_{(p)})$.
Using Lemma \ref{lem:pic-sp} and Proposition \ref{prop:main-variant}, we obtain an $H\Z$-modules structure on $\scr P_{(p)}$.
Consider the following commutative diagram in $\SH$
\begin{equation*}
\begin{tikzcd}
\1 \ar[r] \ar[d, "c"] & H\Z \ar[d, "\tilde c"] \ar[ld, dashed] \ar[ldd, dotted] \\
\scr P \ar[d] \ar[r] & \scr P_{(p)} \ar[d]  \\
\scr P_{\le 1} \ar[r] & \scr P_{(p)\le 1} .
\end{tikzcd}
\end{equation*}
The map $c$ corresponds to $\Sigma^{2,1}\1$.
The map $\tilde c$ is obtained from $c$ via the $H\Z$-module structure on $\scr P_{(p)}$.
We seek to solve the lifting problem indicated by the dashed arrow.
Since the homotopy groups of $\scr P$ are $p$-local in degrees $\ge 2$ we get $\scr P_{>1} \wequi \scr P_{(p)>1}$, or in other words, the vertical fibers in the bottom square are equivalent.
Thus this square is cartesian and our lifting problem is equivalent to the one indicated by the dotted arrow.

Let $\scr Q \in \SH_{[0,1]}$ be arbitrary.
From $H\Z \wequi \1/\1_{\ge 1}$ we get \[ \Map(H\Z, \scr Q) \wequi \Map(\1, \scr Q) \times_{\Map(\1_{\ge 1}, \scr Q)} \{0\}. \]
Since $\Map(\1_{\ge 1}, \scr Q) \wequi \Hom_{\Ab}(\Z/2, \pi_1(\scr Q))$ is discrete, we find that $\Map(H\Z, \scr Q) \to \Map(\1, \scr Q)$ is the inclusion of those components vanishing under the composite \[ \pi_0 \scr Q \wequi \pi_0 \Map(\1, \scr Q) \to \pi_0 \Map(\1_{>1}, \scr Q) \wequi \Hom_{\Ab}(\Z/2, \pi_1 \scr Q) \wequi \pi_1(\scr Q)[2]. \]
Under the equivalence between $\SH_{[0,1]}$ and grouplike symmetric monoidal (``Picard'') groupoids, this is precisely the map sending an object to its switch automorphism.\footnote{To see this, since $\1_{\le 1}$ corresponds to the free Picard groupoid on one generator, it suffices to show that the switch map on this generator is nonzero. But this follows from the existence of \emph{any} object in a Picard groupoid with a non-trivial switch, which is clear.}

Coming back to our lifting problem, we find that any dotted arrow making the left hand triangle commute also makes the right hand triangle commute, and that such an arrow exists if and only if the switch map on $\Sigma^{2,1} \1_{(p)} \in \SH(\F_p)_{(p)}$ is homotopic to the identity.
The switch is given by the image of $\lra{-1}$ \cite[Remark 6.3.5]{morel-trieste} under the rank map (recall that $GW(\F_p)_{(p)} \wequi \Z_{(p)}$), which is indeed $1$.
\end{proof}

\begin{rmk}
We have used the assumption that $p \ne 2$ two times in the above proofs: to know that $(\ul{GW}_{(p)})^\times$ is strictly $\A^1$-invariant (in Lemma \ref{lem:pic-sp}), and to know that $\ul\pi_0 (\1_{\SH^{S^1,\fet}(k)}) \wequi \ul{GW}$ (in Proposition \ref{prop:S1-fet}).
It seems likely that the assumption is unnecessary for the first statement, and that the second is false without the assumption.
\end{rmk}

\subsection{Graded multiplicative structures} \label{subsec:graded-mult}
\subsubsection{Classical case}
Let $A$ be a commutative monoid (in sets).
In what follows, already the case $A=*$ is interesting.
Denote by $^A\Fin$ the category of finite sets with a morphism to $A$.
The category $^A\Fin$ is extensive, and so $\Span({}^A\Fin)$ makes sense.
Note that \[ A \times \Fin \to {}^A\Fin \quad\text{and}\quad A \times \Span(\Fin) \to \Span({}^A\Fin) \] induced by \[ (a, X) \mapsto (X \to \{a\} \hookrightarrow A) \] are the initial functors to a category with finite coproducts, preserving finite coproducts in the second variable; in fact ${}^A\Fin$ is the free finite coproduct cocompletion of $A$.\todo{details...?}
It follows that\NB{details?} \[ \PSh_\Sigma({}^A\Fin) \wequi \Fun(A, \Spc) \quad\text{and}\quad  \PSh_\Sigma(\Span({}^A\Fin)) \wequi \Fun(A, \CMon(\Spc)). \]
Transferring the Day convolution symmetric monoidal structures \cite[Proposition 4.8.1.10]{HA} to the left hand sides and restricting, we obtain symmetric monoidal structures on $^A\Fin$ and $\Span({}^A\Fin)$, informally described as $(X,f) \otimes (Y,g) = (X \times Y,f+g)$.
Then the symmetric monoidal structures on the left hand sides are also induced by Day convolution.

The forgetful functor $\CAlg(\PSh_\Sigma(\Span({}^A\Fin))) \to \PSh_\Sigma(\Span({}^A\Fin))$ admits a left adjoint.
Denote by $\BiSpan^A(\Fin) \subset \CAlg(\PSh_\Sigma(\Span({}^A\Fin)))$ the full subcategory on the essential image of this left adjoint applied to $\Span({}^A\Fin) \subset \PSh_\Sigma(\Span({}^A\Fin))$.
This enjoys the following well-known universality property.
\begin{prop} \label{prop:bispans}
\begin{enumerate}
\item We have $\CAlg(\PSh_\Sigma(\Span({}^A\Fin))) \wequi \PSh_\Sigma(\BiSpan^A(\Fin))$.
\item If $\scr C$ is semiadditive and presentably symmetric monoidal, then there is a canonical functor $\BiSpan^A(\Fin) \to \CAlg(\scr C^A)$, where $\scr C^A$ is given the Day convolution symmetric monoidal structure.
\end{enumerate}
\end{prop}
\begin{proof}
(1) Essentially by construction, $\BiSpan^A(\Fin)$ constitutes a set of compact projective generators closed under finite coproducts\NB{details?}, so this is clear \cite[Proposition 5.5.8.25]{HTT}.

(2) We obtain from Proposition \ref{prop:semiadd-basics} a symmetric monoidal functor $\Span(\Fin) \to \scr C$, which we can extend over $\PSh_\Sigma(\Span(\Fin))$ by the universal property of the Day convolution symmetric monoidal structure.
The result follows by applying $\CAlg((\ph)^A)$ and passing to the subcategory $\BiSpan^A(\Fin)$.
\end{proof}

Denote by $F^+$ the composite \[ \PSh_\Sigma({}^A\Fin) \xrightarrow{L} \PSh_\Sigma(\Span({}^A\Fin)) \xrightarrow{U} \PSh_\Sigma({}^A\Fin) \] (i.e. the free commutative monoid functor), by $F^\times$ the composite \[ \PSh_\Sigma({}^A\Fin) \xrightarrow{L} \CAlg(\PSh_\Sigma({}^A\Fin)) \xrightarrow{U} \PSh_\Sigma({}^A\Fin) \] where $\CAlg$ refers to the (in general non-cartesian) symmetric monoidal structure constructed above (i.e. $F^\times$ is the free ``multiplicative'' commutative monoid functor), and by $F_R$ the composite \[ \PSh_\Sigma({}^A\Fin) \xrightarrow{L} \PSh_\Sigma(\Span({}^A\Fin)) \xrightarrow{L} \CAlg(\PSh_\Sigma(\Span({}^A\Fin))) \xrightarrow{U} \PSh_\Sigma({}^A\Fin) \] (i.e. the free commutative semiring functor).
Here functors labelled $U$ forget structure, and functors labelled $L$ are appropriate ``free'' functors (left adjoints of appropriate $U$).

\begin{sch}
We can elucidate the above notions by replacing spaces by sets throughout.
For $F^+$, we find that $\CMon(\PSh_\Sigma({}^A\Fin))_{\le 0} \wequi \PSh_\Sigma(\Span({}^A\Fin))_{\le 0} \wequi \CMon(\Spc_{\le 0}^A)$ is the usual category of $A$-graded commutative monoids: objects are $A$-indexed families of commutative monoids $X_a$, traditionally written as $X = \bigoplus_{a \in A} X_a$.
Elements can be added if they are in the same degree (say $a$) using the monoid structure (on $X_a$), and if they are in differing degrees they can be added formally.

In contrast for $F^\times$ we observe that $\CAlg(\PSh_\Sigma({}^A\Fin))_{\le 0} \wequi \CAlg(\Spc_{\le 0}^A)$ consists of families of sets $X = \prod_{a \in A} X_a$ together with maps $X_a \times X_b \to X_{a+b}$ making $X$ into an abelian monoid.

The definition of the category $\CMon(\Spc_{\le 0}^A)$ does not even use the commutative monoid structure on $A$.
However, using the commutative monoid structure, we can give it a further symmetric monoidal structure.
In the traditional notation it takes the form $X \otimes Y = \bigoplus_{a \in A} \bigoplus_{b+c=a} X_b \otimes X_c$.
The functor $F_R$ is related to commutative algebras for this symmetric monoidal structure.
This is the usual notion of an $A$-graded semiring: we are given sets $X_a$ for $a \in A$ together with addition maps $X_a \times X_a \to X_a$ and multipication maps $X_a \times X_b \to X_{b+c}$ (for $a,b,c \in A$), satisfying certain axioms.
\end{sch}

The name of the category $\BiSpan^A(\Fin)$ is justified by the following.
\begin{lem} \label{lem:bispans}
Let $(X,f) \in {}^A\Fin$ and $(Y,g) \in \PSh_\Sigma({}^A\Fin)$.
\begin{enumerate}
\item We have $F_R \wequi F^+ \circ F^\times$.
\item The space $\Map_{\PSh_\Sigma({}^A \Fin)}(X,F^+ Y)$ is equivalent to the groupoid of spans \[ X \leftarrow (T,p) \rightarrow Y \] with $(T,p) \in {}^A\Fin$.
\item The space $\Map_{\PSh_\Sigma({}^A \Fin)}(X,F^\times Y)$ is equivalent to the groupoid of spans \[ X \xleftarrow{\alpha} (T,p) \rightarrow Y, \] where $(T,p) \in {}^A\Fin$ and $\alpha: T \to X$ is a map of finite sets satisfying $f(x) = \sum_{t \in \alpha^{-1}(x)}p(t)$ (so $\alpha$ is not a map in ${}^A\Fin$!).
\item Assuming $(Y,g) \in {}^A\Fin$, the space $\Map_{\BiSpan^A(\Fin)}(X,Y)$ is equivalent to the groupoid of bispans \[ X \leftarrow (T,p) \xleftarrow{\alpha} (T',p') \to Y, \] where the unlabelled maps are maps in $^A\Fin$ and $\alpha$ is as in (3).
\end{enumerate}
\end{lem}
\begin{proof}
(1) We have a commutative diagram
\begin{equation*}
\begin{CD}
\CAlg(\PSh_\Sigma({}^A\Fin)) @>f>> \CAlg(\PSh_\Sigma(\Span({}^A\Fin))) \\
@A{S}AA                             @A{S}AA \\
\PSh_\Sigma({}^A\Fin) @>f>> \PSh_\Sigma(\Span({}^A\Fin))
\end{CD}
\end{equation*}
induced by the symmetric monoidal functor $^A\Fin \to \Span({}^A\Fin)$ (this follows from the formula for free commutative algebras \cite[Example 3.1.3.14]{HA}, using that $f$ preserves finite coproducts and hence its extension $\PSh_\Sigma$ preserves colimits).
Denote by $f^r$ the right adjoints of $f$, and by $S^r$ the right adjoints of $S$ (all versions of forgetful functors).
Then $f$ commutes with $S^r$\NB{essentially by definition of $f$ ... ref?}, and hence we deduce that \[ F_R = f^rS^rSf \wequi f^rS^rfS \wequi f^rfS^rS. \]
We have $f^rf \wequi F^+$ and $S^rS \wequi F^\times$, essentially by definition.
The result follows.

(2) Clear from the definition of $F^+$ via $\Span({}^A\Fin)$.
(Reduce to the case $Y \in {}^A\Fin$ using that both functors preserve sifted colimits in $Y$.)
Alternatively, the argument for (3) applies with minor modification.

(3) Fix $Y$ and view either expression as a functor of $X$.
Then both preserve finite coproducts.
Since ${}^A\Fin$ is the free cocompletion of $A$ under finite coproducts, it thus suffices to treat the case $X = \{*\}_a$, i.e. a one point set mapping to $a \in A$.
In this case the groupoid of spans just reduces to (the groupoid of) maps $T \xrightarrow{\beta} Y$ with $a = \sum_{t \in T} g(\beta(t))$.
This is a full subgroupoid of the groupoid $M(Y)$ of all maps $T \to Y$ with $T$ a finite set, consisting of those maps satisfying appropriate degree restrictions.
On the other hand by general theory of free algebras \cite[Proposition 3.1.3.13]{HA} and using that $\{*\}_a$ is completely compact in $\PSh_\Sigma({}^A\Fin)$ we have \[ \Map(\{*\}_a, F^\times (Y,g)) \wequi \coprod_{n \ge 0} \Map(\{*\}_a, (Y,g)^{\otimes n})_{h\Sigma_n}. \]
There is an evident fully faithful map from this groupoid to $M(Y)$ (sending a point $(y_1, \dots, y_n)$ to $(\{1, \dots, n\} \to Y, i \mapsto y_i)$).
One checks easily that the two subgroupoids of $M(Y)$ agree, and hence we obtain the required isomorphism functorially in $X$ (and also $Y$).

(4) Immediate from (1), (2), (3).
\end{proof}

\begin{exm}
We illustrate Lemma \ref{lem:bispans}.
One expects that $\pi_0 \Map(\{*\}_a, F_R X)$ is the same as morphisms of commutative graded semirings from $\NN[t_a]$ (with generator $t_a$ in degree $a$) to the free commutative graded semiring $C$ on $X$.
This should be the same as elements of $C_a$, which themselves are sums of monomials in $X$ of degree $a$, which themselves are products of elements of $X$ of degrees adding up to $a$.
This is precisely the set of isomorphism classes of bispans from $\{*\}_a$ to $X$ described above.
\end{exm}

\begin{exm}
If $A=*$ then $F^+ \wequi F^\times$.
This simplifies the proof of Lemma \ref{lem:bispans}(3) in this case.
\end{exm}

\subsubsection{Motivic case: construction of norms}
Now we consider the motivic analog.
Denote by $^A\Sch_S$ the category of $S$-schemes together with a continuous (i.e. locally constant) function to $A$.
Let $\scr L$ denote a class of smooth morphisms in $\Sm_S$ closed under composition, base change, finite coproducts, finite étale extension and finite étale Weil restriction.
Denote by $\Cor^\fet({}^A\scr L_S)$ the $(2,1)$-category with objects those objects in $^A\Sch_S$ whose structural morphism to $S$ lies in $\scr L$, and morphisms the spans $X \xleftarrow{p} Y \to Z$, where $p$ is finite étale.
As before we get \[ \PSh_\Sigma({}^A\scr L_S) \wequi \PSh_\Sigma(\scr L_S)^A \quad\text{and}\quad \PSh_\Sigma(\Cor^\fet({}^A\scr L_S)) \wequi \PSh_\Sigma(\Cor^\fet)^A. \]
\begin{exm}
We can take $\scr L = \SmQP$ the class of smooth and quasi-projective morphisms, or $\scr L = \SmQA$ the class of smooth and quasi-affine morphisms.
We can also take $\scr L = \SmAff$ the class of smooth \emph{affine} morphisms, but that may not be very reasonable unless $S$ is separated.\footnote{E.g. if $S$ is not separated, then a morphism from an affine scheme to $S$ is not automatically affine, and so when using descent arguments it may not be possible to reduce to the case of affine schemes.}
\end{exm}

In what follows, we also assume that $\scr L$ consists of quasi-projective morphisms and that being in $\scr L$ is finite étale local on the source; this holds in all the above examples.

The proof of Theorem \ref{thm:Spcfet} shows that there is a natural map $\Cor^\fet(\scr L_S) \to \SH(S)$ and hence $\Cor^\fet({}^A\scr L_S) \to \SH(S)^A$.
We now make this compatible with norms.
In fact, as in the classical case, there are in general two different normed structures on $^A\scr L_S$ and the related categories, which we think of as corresponding to product and tensor product.
Similarly there are two normed structures on $\SH(\ph)^A$: a ``pointwise'' one, and an ``interesting'' one (using the commutative monoid structure on $A$) constructed in \cite[\S13.3]{norms}.
While below we will only be interested in the ``interesting'' case, we record the pointwise one for completeness.

To begin with, on the category $^A\scr L_S$ we have the cartesian symmetric monoidal structure $\times$, and also the symmetric monoidal structure given by $(X,f) \otimes (Y,g) = (X \times Y, f + g)$.
Both of these are functorial in $S$ and hence yield \[ ^A\scr L_{(\ph)}^\times, ^A\scr L_{(\ph)}^\otimes: \Sch^\op \to \CMon(\Cat_\infty). \]
Both functors satisfy finite étale descent (finite étale morphisms being of effective descent for quasi-projective schemes \cite[Exposé VIII, Corollaire 7.7]{SGA1}) and hence promote to functors \cite[Corollary C.13]{norms} \[ ^A\scr L_{(\ph)}^\times, ^A\scr L_{(\ph)}^\otimes: \Span(\Sch, \all, \fet) \to \Cat_\infty. \]
The construction $\Cor^\fet(\ph) = \Span(\ph, \fet, \all)$ is functorial and hence we can build \[ \Cor^\fet({}^A\scr L_{(\ph)}^\times), \Cor^\fet({}^A\scr L_{(\ph)}^\otimes): \Span(\Sch, \all, \fet) \to \Cat_\infty. \]

On either functor we can perform the motivic construction sectionwise.
In the first case we invert objects which are spheres in all degrees, and obtain a transformation \[ \PSh_\Sigma(\Cor^\fet({}^A\scr L_{(\ph)}^\times)) \to \SH(\ph)^{A\times}, \] where on the right hand side we mean the pointwise normed structure on $\SH(\ph)^A$.
See \cite[\S3.3.3]{bachmann-MGM} for details.
In the second case we invert spheres in degree $0$ and obtain a transformation \[ \PSh_\Sigma(\Cor^\fet({}^A\scr L_{(\ph)}^\otimes)) \to \SH(\ph)^{A\otimes}; \] for details see \cite[\S13.3]{norms} (in particular, $\otimes$-inverting spheres in degree $0$ has the effect of inverting spheres in all degrees in the usual way).

We can pass to normed objects, as in \cite[\S7]{norms} and get \begin{equation} \label{eq:NAlg-funct} \NAlg_{\scr L_S}(\PSh_\Sigma(\Cor^\fet({}^A\scr L_{(\ph)}^\otimes))) \to \NAlg_{\scr L_S}(\SH(\ph)^{A\otimes}), \end{equation} and similarly for the pointwise structures.

\subsubsection{Motivic case: free semiring objects}
Denote by \[ \BiSpan^A(\scr L_S) \subset \NAlg_{\scr L_S}(\PSh_\Sigma(\Cor^\fet({}^A\scr L_{(\ph)}^\otimes))) \] the full subcategory on the free normed objects on objects in $\scr L_S$.

We begin with the following very formal result, somewhat analogous to Proposition \ref{prop:bispans}.
\begin{prop} \label{prop:bispans-mot}
\begin{enumerate}
\item We have $\NAlg_{\scr L_S}(\PSh_\Sigma(\Cor^\fet({}^A\scr L_{(\ph)}^\otimes))) \wequi \PSh_\Sigma(\BiSpan^A(\scr L_S))$.
\item There is a canonical functor $\BiSpan^A(\Fin) \to \BiSpan^A(\scr L_S)$ extending the symmetric monoidal coproduct preserving functor $\Fin \to \scr L_S$.
\end{enumerate}
\end{prop}
\begin{proof}
(1) Entirely analogous to Proposition \ref{prop:bispans}(1).

(2) Compose the functor $\BiSpan^A(\Fin) \to \CAlg(\PSh_\Sigma(\Cor^\fet(\scr L_S))^A)$ coming from Proposition \ref{prop:bispans}(2) with the left adjoint of $\NAlg_{\scr L_S}(\PSh_\Sigma({}^A\Cor^\fet(\scr L_{(\ph)}^\otimes))) \to \CAlg(\PSh_\Sigma(\Cor^\fet({}^A\scr L_S)))$.
\end{proof}

Denote by $F^+$ the composite \[ F^+: \PSh_\Sigma({}^A\scr L_S) \xrightarrow{L} \PSh_\Sigma(\Cor^\fet({}^A\scr L_S)) \xrightarrow{U} \PSh_\Sigma({}^A\scr L_S), \] by $F^\times$ the composite \[ F^\times: \PSh_\Sigma({}^A\scr L_S) \xrightarrow{L} \NAlg_{\scr L_S}(\PSh_\Sigma({}^A\scr L_{(\ph)}^{\otimes})) \xrightarrow{U} \PSh_\Sigma({}^A\scr L_S) \] and by $F_R$ the composite \[ F_R: \PSh_\Sigma({}^A\scr L_S) \xrightarrow{L} \NAlg_{\scr L_S}(\PSh_\Sigma(\Cor^\fet({}^A\scr L_{(\ph)}^\otimes))) \xrightarrow{U} \PSh_\Sigma({}^A\scr L_S), \] where in each case $U$ denotes a functor forgetting structure and $L$ is its left adjoint.

\begin{rmk} \label{rmk:F+}
Under the equivalence $\PSh_\Sigma({}^A\scr L_S) \wequi \PSh_\Sigma(\scr L_S)^A$, essentially by definition the functor $F^+$ is given by applying the ungraded version $F$ degreewise.
\end{rmk}

Here is a motivic analog of Lemma \ref{lem:bispans}.
\begin{prop} \label{prop:Spcfet-mult}
Let $(X,f) \in {}^A\scr L_S$ and $(Y,g) \in \PSh_\Sigma({}^A\scr L_S)$.
\begin{enumerate}
\item We have $F_R \wequi F^+ \circ F^\times$.
\item The space $\Map((X,f),F^+ (Y,g))$ is equivalent to the groupoid of spans \[ X \leftarrow (T,p) \to Y \] such that $(T,p) \in {}^A\scr L_S$, the backwards map is finite étale, and both maps are in $\PSh_\Sigma({}^A \scr L_S)$.
\item The space $\Map((X,f),F^\times (Y,g))$ is equivalent to the groupoid of spans \[ X \xleftarrow{\alpha} (T,p) \to Y \] such that $(T,p) \in {}^A\scr L_S$, $\alpha$ is a finite étale morphism of the underlying (ungraded) schemes, $(T,p) \to (Y,g)$ is a morphism in $\PSh_\Sigma({}^A\scr L_S)$ and $f(\bar x) = \sum_{\bar t \in \alpha^{-1}(\bar x)} p(\bar t)$ for every geometric point $\bar x$ of $X$ (in particular $\alpha$ is generally \emph{not} a morphism in ${}^A\scr L_S$).
\item Assuming $(Y,g) \in {}^A\scr L_S$, the space $\Map_{\BiSpan^A(\scr L_S)}(X,Y)$ is equivalent to the groupoid of bispans \[ X \leftarrow (T,p) \xleftarrow{\alpha} (T',p') \to Y \] with backwards arrows finite étale, unlabelled arrows in $^A\scr L_S$ and $\alpha$ as in (3).
\end{enumerate}
\end{prop}
All of these equivalences are natural in $S, X$ and $Y$ in the evident way.
\begin{proof}
The proofs of (1), (2) and (4) are entirely analogous to Lemma \ref{lem:bispans}.

The proof of (3) is also essentially analogous.
The main non-trivial input is a formula for free normed objects for which see \cite[Theorem 3.10]{bachmann-MGM}.\todo{elaborate on this when reference available}
The other thing to observe is that we seek to establish an equivalence between $\Sigma$-presheaves, which we may check stalkwise \cite[Lemma 2.6]{norms}, and hence we may reduce to the case where $X$ is connected (but possibly only essentially smooth over $S$).
\end{proof}

\subsection{Normed spectra and simplicial commutative rings} \label{subsec:normed-SCR}
In order to discuss norms, it is convenient to model $\Spc(S)$ by $\PSh(\SmAff_S)$ instead of $\PSh(\Sm_S)$.
For this we need to assume that $S$ is separated.
In fact we shall assume $S$ affine.
In the language of the previous subsection we denote by $\scr L$ the class of smooth affine morphisms.

Fix a commutative monoid (in sets) $A$.
Let us put $\NAlg_\SmAff(\SH(S)^A) := \NAlg_{\SmAff_{S}}(\SH(\ph)^{A\otimes})$ and similarly for $\NAlg_\SmAff(\PSh_\Sigma(\Cor^\fet({}^A\SmAff_S)))$.
We also write $\BiSpan^A(\FEt_S) \subset \BiSpan^A(\SmAff_S)$ for the full subcategory corresponding to ${}^A\FEt_S \subset {}^A\SmAff_S$.
It follows from \cite[Proposition C.1]{bachmann-splittings}\todo{ref when available} that the formation of these categories is functorial in $S$.\footnote{This is where we use that $S$ is affine and not just separated.}
Using this we see that the functor of \eqref{eq:NAlg-funct} is natural in $S$ and hence yields (via Construction \ref{cnstr:enriched}) \[ \ul{\BiSpan^A(\FEt_S)} \hookrightarrow \ul{\NAlg_\SmAff(\PSh_\Sigma(\Cor^\fet(\SmAff_S))^A)} \to \ul{\NAlg_\SmAff(\SH(S)^A_{(p)})} \in \Cat^{\PSh(\SmAff_S)}_\infty. \]

We make use of the localization $L$ of $\PSh(\SmAff_S)$ with weak equivalences detected by $\PSh(\SmAff_S) \to \SH(S)_{(p)}$.
One checks easily that all result of \S\ref{sec:main} regarding $L$ also hold in the affine setting.

\begin{lem} \label{lem:NAlg-SHp-uMap-local}
Let $S$ be an affine scheme.
The $\PSh(\SmAff_S)$-valued mapping spaces in $\ul{\NAlg_\SmAff(\SH(S)_{(p)}^A)}$ are $L$-local.
\end{lem}
\begin{proof}
Formation of $\PSh(\SmAff_S)$-valued mapping spaces turns colimits in the first variable into limits (because base change of normed spectra preserves colimits\todo{ref when available}).
Since $L$-local objects are closed under limits and $\NAlg_\SmAff(\SH(S)^A)$ is generated under colimits by free normed spectra (the forgetful functor to $\SH(S)^A$ is conservative by construction)\NB{details?}, we reduce to proving that the mapping presheaves in $\ul{\SH(\ph)^A_{(p)}}$ are $L$-local.
Now if $E_\bullet, F_\bullet \in \SH(S)^A_{(p)}$ then for $a \in A$ we have $E_a, F_a \in \SH(S)_{(p)}$ and also \[ \ul\Map(E_\bullet, F_\bullet) \wequi \prod_{a \in A} \ul\Map(E_a, F_a). \]
Since $L$-local objects are closed under products, the claim follows from Lemma \ref{lem:SHp-uMap-local}.
\end{proof}

\begin{prop} \label{prop:mult-maps}
Let $S$ be an affine, quasi-compact $\F_p$-scheme.
Let $X, Y \in \BiSpan^A(\FEt_S)$.
The natural map \[ \ul{\Map(X,Y)} \to a_\et \pi_0 \ul{\Map(X,Y)} \] is an $L$-equivalence.\NB{It seems that $\scr F \to a_\et \pi_0 \scr F$ is an $L$-equivalence for many other étale sheaves $\scr F$ as well.}
\end{prop}
\begin{proof}
Both sides convert finite disjoint unions in the first variable into products.
Since $L$-equivalences are stable under finite products (Lemma \ref{lem:L-basics}), we may thus assume that $X = (X,a)$ for some constant function $a \in A$.\footnote{This is where we use that $S$, and hence $X$, is quasi-compact.}

For $F \in \PSh_\Sigma({}^A\SmAff_S) \wequi \PSh_\Sigma(\SmAff_S)^A$ and $a \in A$, denote by $F_a \in \PSh_\Sigma(\SmAff_S)$ the component at $a$.
Let $p: X \to S$ be the structure map.
Then one finds that $\ul\Map(X,Y) \wequi p_*p^* F_R(Y)_a$, where $F_R \wequi F^+ \circ F^\times$ is the functor from Proposition \ref{prop:Spcfet-mult}.
Since $p$ is finite étale, formation of $a_\et \pi_0$ commutes with $p_*p^*$ (on $\Sigma$-presheaves).
Since $p_*p^*$ preserves $L$-equivalences (by Lemma \ref{lem:L-basics}) we may hence assume that $X=S$.
In other words we need to prove that $F_R(Y)_a \to a_\et \pi_0 F_R(Y)_a$ is an $L$-equivalence.
Recall from Lemma \ref{lem:bispans} and Remark \ref{rmk:F+} that \[ F_R(Y)_a \wequi F(F^\times(Y)_a) \in \PSh_\Sigma(\SmAff_S), \] where $F$ is the $\PSh_\Sigma$ version of the ungraded free normed space functor.
We shall prove the following.
\begin{enumerate}
\item $F^\times(Y)_a$ is $L$-equivalent to an ind-(finite étale) scheme, namely $a_\et \pi_0 F^\times(Y)_a$.
\item $F$ preserves $L$-equivalences.
\item For an ind-(finite étale) scheme $Z$, $F(Z) \wequi a_\et \pi_0 F(Z)$.
\item For any presheaf $G$, we have $a_\et \pi_0 F(a_\et \pi_0 G) \wequi a_\et \pi_0 F(G)$.
\end{enumerate}
Combining (1) and (2) we see that $F(F^\times(Y)_a)$ is $L$-equivalent to $F(a_\et \pi_0 F^\times(Y)_a)$, which by (1) and (3) is $L$-equivalent to $a_\et \pi_0F(a_\et \pi_0 F^\times(Y)_a)$.
This is the same as $a_\et \pi_0 F(F^\times(Y)_a)$, by (4).

It remains to prove the claims.
(2) follows from the fact that $F$ is a colimit of functors of the form $f^*, f_\sharp, f_\otimes$ \cite[Theorem 3.4]{bachmann-MGM}.
Since $a_\et \pi_0 F$ factors through étale sheafification, and the étale sheafification of $F$ is just the ordinary free commutative monoid functors, (4) is straightforward.
(3) is Remark \ref{rmk:alt-arg} (plus compatibility with filtered colimits).
Finally we prove (1). Using (3) and Lemma \ref{lem:summand} below, it will suffice to exhibit $F^\times(Y)_a$ as a summand of $F(Y)$.
Using that $F=F^\times$ in the ungraded setting, forgetting the grading induces a map \[ L_\Sigma \coprod_{a \in A} F^\times(Y)_a \to F(Y) \in \PSh_\Sigma(\SmAff_S). \]
Using Proposition \ref{prop:Spcfet-mult}(3) it is verified to be an equivalence.
This concludes the proof.
\end{proof}

In the proof we used the following.
\begin{lem} \label{lem:summand}
Let $f: X \to Y$ and $g: X' \to Y'$ be maps in $\Spc(S)$.
Then $f \amalg g$ is an $L$-equivalence if and only if $f$ and $g$ are.
\end{lem}
\begin{proof}
Sufficiency is clear.
For necessity, suppose that $f \amalg g$ is an $L$-equivalence.
We need to prove that $\Sigma^\infty_+(f)_{(p)}$ is an equivalence.
This is true since it is a retract of the equivalence $\Sigma^\infty_+(f \amalg g)_{(p)}$.
\end{proof}

\begin{cor} \label{cor:SCR-A}
Let $S$ be an $\F_p$-scheme and $A$ a commutative monoid (in sets).
The canonical functor $\BiSpan^A(\Fin) \to \NAlg(\SH(S)^A_{(p)})$ factors through $\h\BiSpan^A(\Fin)$.
\end{cor}
\begin{proof}
Since $\NAlg(\SH(S)^A)$ is functorial in $S$ \cite[Corollary C.3]{bachmann-splittings}\todo{ref when available}, we may assume that $S = \Spec(\F_p)$, which is affine and noetherian.
We use the notation from Remark \ref{rmk:alt-arg}.
By Proposition \ref{prop:mult-maps}, the canonical map \[ \ul{\BiSpan^A(\FEt_S)} \to L_\et^0 \ul{\BiSpan^A(\FEt_S)} \] is an $L$-equivalence in $\Cat^{\PSh(\SmAff_S)}_\infty$.
It follows (via Lemma \ref{lem:mapping-space-loc}) that we obtain a factorization \[ \BiSpan^A(\FEt_S) \to L_\et^0 \BiSpan^A(\FEt_S) \to \NAlg(\SH(S)^A_{(p)}), \] where the category in the middle has discrete mapping spaces.
Thus $\BiSpan^A(\Fin) \to \BiSpan^A(\FEt_S) \to L_\et^0 \BiSpan^A(\FEt_S)$ factors through $\h\BiSpan^A(\Fin)$, concluding the proof.
\end{proof}

\begin{cor} \label{cor:SCR}
Let $S$ be an $\F_p$-scheme.
The canonical functor $\BiSpan(\Fin) \to \NAlg(\SH(S)_{(p)})$ factors through $\h\BiSpan(\Fin)$.
\end{cor}
\begin{proof}
This is the special case $A=*$.
\end{proof}

\begin{rmk}
The category $\h\BiSpan(\Fin)$ identifies with the category of free finitely generated semirings (i.e. semirings of the form $\NN[t_1, \dots, t_n]$).
The category of additively grouplike objects in $\PSh_\Sigma(\h\BiSpan(\Fin))$ thus identifies with the usual $\infty$-category of ``simplicial commutative rings'' or (the opposite of) ``derived affine schemes''; see e.g. \cite[Chapter 2.2]{HAGII} \cite[Chapter 25]{SAG} \cite{lurie-thesis} for more about these.

The right adjoint of $\PSh_\Sigma(\BiSpan(\Fin)) \to \NAlg(\SH(S))$ constructs an $\scr E_\infty$-ring space underlying any normed spectrum.
The factorization of $\PSh_\Sigma(\BiSpan(\Fin)) \to \NAlg(\SH(S)_{(p)})$ through (grouplike objects in) $\PSh_\Sigma(\h\BiSpan(\Fin))$ in the case $S$ is an $\F_p$-scheme upgrades the underlying $\scr E_\infty$-ring space to a simplicial commutative ring.
\end{rmk}

\subsection{Application to algebraic cobordism} \label{subsec:MGL}
\begin{lem} \label{lem:graded-e}
Let $A$ be a small symmetric monoidal $\infty$-category and $\scr C, \scr D$ presentably symmetric monoidal $\infty$-categories.
Then cocontinuous symmetric monoidal functors $\scr C^A \to \scr D$ are the same as pairs of a symmetric monoidal functor $A^\op \to \scr D$ and a cocontinuous symmetric monoidal functor $\scr C \to \scr D$.
\end{lem}
\begin{proof}
Using \cite[Proposition 4.8.1.17]{HA} one checks that $\PSh(A^\op) \otimes \scr C \wequi \Fun(A, \scr C)$\todo{really??}; here $\otimes$ denotes the tensor product on the category $\PrL$ of presentable $\infty$-categories.
Since $\otimes$ is the coproduct in the category $\CAlg(\PrL)$ of presentably symmetric monoidal $\infty$-categories \cite[Proposition 3.2.4.7]{HA}, the result follows.
\end{proof}
Denote by $\Ss \in \CMon(\Spc)^\gp$ the free grouplike commutative monoid on one generator (in degree $0$); we may view this in particular as a small symmetric monoidal $\infty$-category.
If $\scr C$ is a stable, presentably symmetric monoidal $\infty$-category and $L \in \scr C$ is an invertible object, then there is a canonical adjunction \[ e_{\Ss}: \SH^{\Ss} \adj \scr C: e_\Ss^* \] where $e_\Ss$ is symmetric monoidal and sends the object represented by the generator $1 \in \Ss$ to $L$ (use Lemma \ref{lem:graded-e}).
Soon we shall specialize to the case $\scr C = \SH(\F_p)_{(p)}$ and $L = \Sigma^{2,1}\1$.
Before we record the following general observation.
\begin{lem}\label{lem:period-crit}
Let $\scr C$ be a presentably symmetric monoidal $\infty$-category and $L \in \scr C$ invertible.
Let $E \in \CAlg(\scr C)$.
Then $\coprod_{n \in \Z} L^{\otimes n} E$ (with its obvious multiplication map) lifts to $\CAlg(\scr C^\Z)$ if and only if the canonical map $\Ss \to \Pic(\Mod_E(\scr C)) \in \CMon(\Spc)$ factors through $\Z$.
\end{lem}
An object $E$ as above is called \emph{periodizable}.
\begin{proof}
Let $\tilde E \in \CAlg(\scr C^\Z)$ be a lift.
Restriction followed by left Kan extension along $\{0\} \to \Z$ makes $\tilde E$ into an algebra under $F$, where $F$ can be described informally as $F_0 = \tilde E_0 \wequi E$ and $F_i = 0$ else.
Observe that $\Mod_{F}(\scr C^\Z) \wequi (\Mod_E(\scr C))^\Z$.\NB{details?}
We find that $\tilde E$ promotes to an object of $\CAlg((\Mod_E(\scr C))^\Z)$, i.e. a lax symmetric monoidal functor $\Z \to \Mod_E(\scr C)$ \cite[Proposition 2.12]{glasman2013day}.
The functor is actually strong (by construction and assumption, the lax witness maps $E \otimes_E E \to E$ are given by multiplication), proving the first half of the claimed equivalence.
Conversely, given $\Z \to \Pic(\Mod_E(\scr C))$ we can compose with $\Pic(\Mod_E(\scr C)) \to \Mod_E(\scr C)$ to obtain a symmetric monoidal functor $\Z \to \Mod_E(\scr C)$, whence (by \cite[Proposition 2.12]{glasman2013day} again) an object of $\CAlg((\Mod_E(\scr C))^\Z)$.
It is easily verified to have the claimed properties.

\end{proof}

\begin{exm}
If $S$ is an $\F_p$-scheme and $p \ne 2$, then by Corollary \ref{cor:HZ-Pic} any $E \in \CAlg(\SH(S)_{(p)})$ can be $(2,1)$-periodized.
\end{exm}
\begin{rmk} \label{rmk:periodic-lift-compat}
The proof of Lemma \ref{lem:period-crit} shows that if $\tilde E$ is a periodic lift of $E$ and $\Z \to \Pic(\scr C)$ is the induced functor, then $\tilde E \wequi e_\Z^*(E)$.
\end{rmk}

For a scheme $X$ we have a functor \[ \Vect(X)^\wequi \to \Shv_\Nis(\Sm_{X+})_{\le 0}^\Z, V \mapsto (n \mapsto V_n/V_n \setminus 0), \] where for a vector bundle $V$ on $X$ we denote by $V_n$ the restriction of $V$ to the (clopen) subscheme where $V$ has rank $n$.
Using \cite[Lemma 3.6 and Proposition 3.13]{norms} this is readily made functorial in $X \in \Span(\Sm_S, \all, \fet)$ and so we obtain \[ \Vect(\ph)^\wequi \to \Shv_\Nis(\Sm_{(\ph)})_*^\Z \in \Fun(\Span(\Sm_S, \all, \fet), \CMon(\Cat_\infty)). \]
Composing with motivic localization and stabilization, group completing and Zariski sheafifying the source, we obtain \[ j_\Z: K \to \SH(\ph)^\Z \in \Fun(\Span(\Sm_S, \all, \fet), \Spc). \]
See \cite[\S16.2]{norms} for a closely related construction.
Applying the motivic Thom spectrum functor (i.e. taking the motivic colimit of this diagram in the sense of \cite[\S2]{bachmann-colimits}, which preserves normed objects by \cite[\S3.4]{bachmann-colimits}) we obtain an object \[ \PMGL \in \NAlg(\SH(S)^\Z) \] lifting the normed spectrum $\oplus_n \Sigma^{2n,n} \MGL$ of \cite[Theorem 16.19]{norms}.
Forgetting some structure, we can also view this as an object of $\CAlg(\SH(S)^\Z)$.
\begin{lem} \label{lem:PMGL-twice}
Let $S$ be a scheme.
Denote by $F$ the canonical functor $\SH^\Z \to \Mod_\MGL(\SH(S))^\Z$, by $F^*$ its right adjoint and by $U: \SH^\Z \to \SH^\Ss$ the forgetful functor.
Then we have \[ e_\Ss^*(\MGL) \wequi UF^* \PMGL \in \CAlg(\SH^\Ss). \]
\end{lem}
\begin{proof}
Consider the diagram of cocontinuous symmetric monoidal functors
\begin{equation*}
\begin{tikzcd}
\SH^\Ss \ar[r, "(e)^\Ss"] \ar[dr, "e_\Ss" swap] & \SH(S)^\Ss \ar[r, "(\otimes \MGL)^\Ss"] \ar[d, "e'_\Ss"] &[3em] \Mod_\MGL(\SH(S))^\Ss \ar[d, "e'_\Ss"] \ar[r, "p_!"] & \Mod_\MGL(\SH(S))^\Z \ar[dl, "G"] \\
            & \SH(S)   \ar[r, "\otimes \MGL"] & \Mod_\MGL(\SH(S)).
\end{tikzcd}
\end{equation*}
The functors $(\ph)^\Ss$ are obtained by applying $(\ph)$ levelwise; here $e: \SH \to \SH(S)$ is the unique symmetric monoidal cocontinuous functor.
The functors $e_\Ss$, $e'_\Ss$ are all induced (via Lemma \ref{lem:graded-e}) by $\Sigma^{2,1} \1 \in \Pic(\SH(S))$.
The functor $p_!$ is left Kan extension along $p: \Ss \to \Z$.
The functor $G$ is obtained from the existence of $\PMGL$ by Lemma \ref{lem:period-crit}.
The diagram commutes by several applications of Lemma \ref{lem:graded-e}.
Applying $G^*$ (the right adjoint of $G$) to $\MGL$ yields $\PMGL$ by Remark \ref{rmk:periodic-lift-compat}.
Since the composite of the top row is left adjoint to $UF^*$, the desired result follows from commutativity of the diagram of right adjoints.
\end{proof}

\begin{cor}
Let $S$ be an $\F_p$-scheme.
\begin{enumerate}
\item The object $\Omega^\infty e_\Ss^*(\MGL_{(p)}) \in \CAlg(\Spc^\Ss)$ has the structure of a $\Z$-graded simplicial commutative ring (i.e. an object of $\PSh_\Sigma(\h\BiSpan^\Z(\Fin))$).
\item If we denote by $L_* \in \CAlg(\SH^\Ss_{\le 0}) \wequi \CAlg(\Ab^\Z)$ the Lazard ring, then there is a canonical homotopy class of morphisms $L_* \to e_\Ss^*(\MGL_{(p)})$.
\end{enumerate}
\end{cor}
\begin{proof}
(1) Consider the commutative diagram
\begin{equation*}
\begin{CD}
\CMon(\Spc^\Ss) @>a>> \CMon(\Spc^\Z) \\
@VVV              @VbVV   \\
\SH^\Ss @>>> \SH^\Z @>c>> \SH(S)^\Z.
\end{CD}
\end{equation*}
All the functors are cocontinuous and symmetric monoidal (induced by $\CMon(\Spc) \to \SH$, $\Ss \to \Z$ and $\SH \to \SH(S)$); we denote their right adjoints by $a^*$ and so on.
Lemma \ref{lem:PMGL-twice} shows that \[ \Omega^\infty e_\Ss^*(\MGL_{(p)}) \wequi a^* b^* c^* (\PMGL) \in \CAlg(\CMon(\Spc^\Ss)). \]
The composite \[ \BiSpan^\Z(\Fin) \to \CAlg(\CMon(\Spc^\Ss)) \xrightarrow{cba} \NAlg(\SH(S)^\Z_{(p)}) \] factors through $\h\BiSpan^\Z(\Fin)$ by Corollary \ref{cor:SCR-A}.
The claim follows since $\PMGL$ lifts to $\NAlg(\SH(S)^\Z)$.

(2) Since $\pi_{2*,*} \MGL_{(p)}$ carries a formal group law we obtain a ring homomorphism $r: L_* \to \pi_{2*,*} \MGL_{(p)}$.
As a $\Z$-graded simplicial commutative ring $L_*$ is free (on a non-canonical set of generators), and hence $r$ lifts (uniquely up to homotopy) to \[ \tilde r: L_* \to \Omega^\infty e_\Ss^*(\MGL_{(p)}) \in \PSh_\Sigma(\h\BiSpan^\Z(\Fin)). \]
The desired result follows by forgetting structure and adjunction.
\end{proof}

\begin{exm}
We can express the motivic Brown--Peterson spectrum (at the implicit prime $p$) as \[ \BPGL \wequi \MGL_{(p)} \otimes_{e_\Ss(L_*)} e_\Ss(\BP_*) \in \SH(\F_p)_{(p)}. \]
In particular, this admits an $\scr E_\infty$-ring structure.
Similar observations apply to the motivic Morava $K$-theory spectra.

Again these $\scr E_\infty$-structures can also be obtained as easy consequences of an affirmation of the Hopkins--Morel problem (arguing similarly to \cite[Example 1.5]{bachmann-nilpotence}).
\end{exm}


\bibliographystyle{alpha}
\bibliography{bibliography}

\end{document}